\documentclass[10pt]{article}  
\usepackage{amssymb}
\usepackage{amsmath,amsbsy,amsfonts,amsthm}
\usepackage{graphicx}
\usepackage{geometry}
\usepackage{wrapfig}
\usepackage{float}
\usepackage{subfig}
\usepackage{setspace}

\newtheorem{theorem}{Theorem}[section]
\newtheorem{lemma}[theorem]{Lemma}
\newtheorem{proposition}[theorem]{Proposition}

\newtheorem{remark}[theorem]{Remark}
\newtheorem{definition}[theorem]{Definition}

   \newcommand{\R}{\mathbb{R}}
   \newcommand{\spn}[1]{\langle#1\rangle}

\begin{document}

\numberwithin{equation}{section}

\title{Soliton Solutions to the Curve Shortening Flow on the 2-dimensional hyperbolic plane}

\author{Fábio Nunes da Silva\footnote{Universidade de Bras\'ilia, Department of Mathematics, 70910-900, Bras\'ilia-DF, Brazil, fabionuness@ufob.edu.br} \qquad Keti Tenenblat\footnote{ Universidade de Bras\'{\i}lia,
 Department of Mathematics,
 70910-900, Bras\'{\i}lia-DF, Brazil, K.Tenenblat@mat.unb.br Partially supported by CNPq Proc. 312462/2014-0, Ministry of Science and Technology, Brazil  and FAPDF/Brazil grant 0193.001346/2016.}
  }
\date{}

\maketitle{}

\begin{abstract}
We show that a curve is a soliton solution to the curve shortening flow if and only if its geodesic curvature can be written as the inner product between its tangent vector field and a fixed vector of the 3-dimensional Minkowski space. We use this characterization to provide a qualitative study of the solitons. We show that for each fixed vector there is a 2-parameter family of soliton solution to the curve shortening flow on the 2-dimensional hyperbolic space. Moreover, we prove that each soliton is defined on the entire real line, it is embedded and its geodesic curvature converges to a constant at each end. \newline
\newline
\vspace{0.2cm} 
\newline
\end{abstract}

\section{Introduction}
A family  of curves  $\hat{X}^{t}:I \longrightarrow M$, $t \in [0,T)$, 
on a 2-dimensional Riemannian manifold $M^{2}$  is  said to be a solution 
to the {\it Curve Shortening Flow}  (CSF) with initial condition 
 $\hat{X}^{0}(\cdot)=X(\cdot)$, if it satisfies the following equation
\begin{equation}\label{i1}
\left\{\begin{array}{ll}
\displaystyle{\frac{\partial }{\partial t}\hat{X}^{t}(\cdot)}=\hat{k}^{t}(\cdot)\hat{N}^{t}(\cdot)\\
\hat{X}^{0}(\cdot)=X(\cdot),
\end{array}\right.
\end{equation}
where $\hat{k}^{t}(\cdot)$ 
is the geodesic curvature and 
 $\hat{N}^{t}(\cdot)$ is the unit vector field normal to
 $\hat{X}^{t}(\cdot)$ for each $t \in [0,T)$.

 Epstein and Gage \cite{Epstein} showed that  when  $M^{2}=\R^{2}$, the CSF is 
 geometrically the same if tangential components are added to the right hand side of  
the differential equation  \eqref{i1}.  Therefore, one can define that a 1-parameter 
family of curves
 $\hat{X}^{t}:I \rightarrow \R^{2}$, $t \in [0,T)$ 
is a solution to the CSF in  $\R^{2}$  with initial condition $\hat{X}^{0}(\cdot)=X(\cdot)$, if it satisfies  
\begin{equation}
\left\spn{\displaystyle\frac{\partial}{\partial t}\hat{X}^{t}(\cdot),\hat{N}^{t}(\cdot)\right}=\hat{k}^{t}(\cdot),
\end{equation}
where $\spn{\cdot,\cdot}$ is the cannonical inner product on  $\R^{2}$, $\hat{k}^{t}(\cdot)$ 
is the curvature and $\hat{N}^{t}(\cdot)$ is the unit vector field normal to 
 $\hat{X}^{t}(\cdot)$ for each $t \in [0,T)$.  The name curve shortening flow is justified by the fact that when the curves of the family 
  $\hat{X}^{t}$ are closed, then the length of the curves decreases along the flow, i.e.,  
 it is a gradient type of flow for the arc length functional. 
  Grayson \cite{Grayson1} observed that the CSF is also known as the {\it curvature flow} or  
  the {\it heat flow for isometric immersions}.  
  
According to Epstein and  Gage \cite{Epstein}, 
the original motivation for studying  equation (\ref{i1}) was to find a new and maybe more natural proof of the existence of closed geodesics on Riemannian manifolds. 
However, the first results in this direction were obtained by  
Grayson \cite{Grayson} em 1989. But equation (\ref{i1}) 
for the Euclidean plane was investigated earlier by several authors 
 \cite{Abresch}, \cite{Epstein1}, \cite{Gage}, \cite{Gage1}, \cite{Gage2} e \cite{Grayson1}.

An important class of solutions to the CSF  are those that evolve by isometries or homotheties. Such solutions are called {\it self-similar solutions } and {\it solitons }
if they evolve just by isometries.  On the Euclidean plane, the straight lines are not affected by the flow and they are considered to be trivial solutions. 
Circles evolve homothetically to a point in  finite time. The {\it 
Grim Reaper} curve given by  the graph of the function $f(s)=ln(cos(s))$ evolves by a flow of  translations. Giga \cite{Giga} proved that this is the unique curve on the plane that evolves   by translations.  An example of a plane curve that evolves by isometries of the plane is 
the {\it yin-yang} spiral.  Abresch-Langer \cite{Abresch} and Epstein-Weinstein \cite{Epstein1} investigated the closed curves, not necessarily simple, that evolve by homotheties. Halldorsson \cite{Halldorsson} concluded the description of 
all self-similar solutions on the plane.

Some authors studied  classes of solutions to the CSF on the plane and proved that 
after a certain time the flow evolves into a  self-similar one.  First, Gage \cite{Gage} \cite{Gage1} showed that closed convex curves 
on $\R^{2}$ evolve into circular curves after a certain time.  Then 
Gage and  Hamilton \cite{Gage2} showed that convex closed curves collapse into a point. 
Grayson \cite{Grayson1} proved  that closed embedded curves evolve to circular curves and then they collapse into a point at a finite time. 
Moreover,  Angenent \cite{Angenent}, under  more general conditions, proved  that the CSF 
evolves in a sense into a self-similar flow, showing the importance of self-similar solutions.  
 
One should point out that, when the ambient space is not the Euclidean plane, there are very few results on self-similar solutions to the CSF.   
In 2015, Halldorsson \cite{Halldorsson1} classified all the self-similar solutions on the Minkowski plane. Dos Reis and Tenenblat \cite{Dos Reis} caracterized  and described all the soliton solutions on the sphere $\mathbb{S}^{2}$. Some results on the CSF for Riemannian manifolds different from the plane can be found  in \cite{Gage2}, \cite{Grayson}, \cite{Ma} and \cite{Zhou}, among others.  Moreover, Angenent, \cite{Angenent1} studied the topology of the closed geodesics on compact surfaces by using the CSF. 

In this paper, we will study the soliton solutions of curve shortening flow on the 2-dimensional hyperbolic space $\mathbb{H}^{2}\subset \R_{1}^{3} $, where $\R_{1}^{3}$ is the 3-dimensional Minkowski space.  

\section{Main Results}


We consider the 3-dimensional Minkowki space as $\mathbb{R}_1^{3}=(\mathbb{R}^{3}, \langle, \rangle )$, where $\mathbb{R}^3$ is the 3-dimensional vector space and $\langle, \rangle$ is the Minkowski metric defined by

$$\langle u,v \rangle=-u_1v_1+u_2v_2+u_3v_3.$$

Let $\displaystyle X:I\subset \mathbb{R} \rightarrow \mathbb{H}^{2}\subset \mathbb{R}_1^{3}$ be a regular curve parametrized by arc length $s$. We denote by $T(s)=X^{\prime}(s)$ the tangent vector field, $N(s)=X(s) \times T(s)$ the unit normal vector field and $k(s)=\langle T^{\prime}(s), N(s)\rangle $ the geodesic curvature of $X$. A one parameter family of curves $\displaystyle \hat{X}: I\times J \rightarrow \mathbb{H}^{2}$ is called a \it curve shortening flow \rm (CSF) with initial condition $X$, if 
\begin{equation}\label{sysCSF}
\left\{ \begin{array}{ll}
\displaystyle \left\langle \frac{\partial}{\partial t}\hat{X}(s,t), \hat{N}(s,t)\right\rangle=\hat{k}(s,t),\\
\hat{X}(s,0)=X(s),
\end{array} \right.
\end{equation}
where $\hat{k}^{t}(\cdot)=\hat{k}(\cdot,t)$ is the geodesic curvature and $\hat{N}^{t}(\cdot)=\hat{N}(\cdot,t)$ is the unit normal vector field of $\hat{X}^{t}(\cdot)=\hat{X}(\cdot,t)$. When $X$ is a geodesic i.e. $k=0$, then the family $\hat{X}^{t}(s)=X(s)$ gives a trivial solution to the CSF. Our goal is to study the case when $\hat{X}^{t}(s)$ evolves  by a  1-parameter family of isometries  of $\mathbb{H}^{2}$.

\begin{definition}\rm
Let	$\displaystyle \hat{X}: I\times J \rightarrow \mathbb{H}^{2}\subset \mathbb{R}_1^{3}$ be a solution to the curve shortening flow \eqref{sysCSF} on $\mathbb{H}^{2}$, with initial condition $\displaystyle X: I \rightarrow \mathbb{H}^{2}$.  We say that $X$ is a \it soliton solution to the curve shortening flow \rm if there is a 1-parameter family of isometries $M(t):\mathbb{H}^{2}\rightarrow \mathbb{H}^{2}$ such that $M(0)=Id$ and $$\hat{X}^{t}(s)=M(t)X(s)$$ for all $t \in J$, where $Id$ is the identity map.
\end{definition} 
We remark that an isometry of $\mathbb{H}^{2}$ is an element of the Lie group $O_1(3)=\{M \in GL(3,\mathbb{R}):M^{T}\epsilon M=\epsilon\}$ that  preserves $\mathbb{H}^{2}$, where $M^{T}$ is the transpose of $M$ and 
$$
\epsilon=\left( \begin{array}{lll}
-1 & 0 & 0\\
0 & 1&0\\
0& 0 & 1
\end{array}\right).
$$

\begin{theorem} \label{c2t1}
	Let $\displaystyle X: I \rightarrow \mathbb{H}^{2}$ be a regular curve parametrized by arc length. Then $X(s)$, $s\in I$, is a soliton solution to the curve shortening flow if, and only if, there is a vector $v \in \mathbb{R}^{3}_1\setminus \{0\}$ such that
	\begin{equation}\label{eq2}
	\langle T(s),v\rangle=k(s),
	\end{equation}
	where $T(s)$ is the unit tangent vector field and $k(s)$ is the geodesic curvature  of $X$.
\end{theorem}
We observe that when $X$ is a geodesic of $\mathbb{H}^{2}\subset  \mathbb{R}^{3}_1$,  
then it is a planar curve and hence there exists a vector $v \in \mathbb{R}^{3}_1\setminus\{0\}$ such that $\displaystyle \langle T,v\rangle=0.$

The following theorem describes the qualitative behaviour of the soliton solutions to the CSF in $\mathbb{H}^{2}.$
\begin{theorem}\label{c2t2}
	For any $v \in \mathbb{R}^{3}_1\setminus\{0\}$, there is a 2-parameter family of non-trivial soliton solutions to the  curve shortening flow  on the 2-dimensional hyperbolic space. Each soliton solution is an embedded curve $ X(s)$ on $\mathbb{H}^{2}$, defined for all $s\in \mathbb{R}$. Moreover, at each end, the curvature function $k(s)$ tends to one of the following constants $\{-1,0,1\}$.
\end{theorem}

\section{Proofs of the main results}
In this section we prove our main results.\newline \newline
\bf Proof of Theorem \ref{c2t1}. \rm Suppose that $X(s)$ is parametrized by arc length $s$. If $X$ is a soliton solution to the CSF, then $\displaystyle \hat{X}^{t}(s)=M(t)X(s)$ is solution to \eqref{sysCSF}, where $M(t)$ is a family of isometries of $\mathbb{H}^{2}$.Taking the derivative of $\displaystyle \hat{X}(s,t)$ at $t$, we have

$$\frac{\partial}{\partial t}\hat{X}(s,t)=M^{\prime}(t)X(s).$$ It follows from definition of the CSF that
\begin{equation*}
\hat{k}(s,t)=\left\langle\frac{\partial}{\partial t}\hat{X}(s,t),\hat{N}(s,t)\right\rangle=\left\langle M^{\prime}(t)X(s),M(t)N(s)\right\rangle.
\end{equation*}
In particular, for $t=0$, we have 
\begin{equation*}
k(s)=\langle M^{\prime}(0)X(s),N(s)\rangle.
\end{equation*}
$M^{\prime}(0)$ is an element of the Lie algebra $\mathfrak{o}_1(3)$ of the Lie group $O_1(3)$. Let ${A_1,A_2,A_3}$ be a basis of $\mathfrak{o}_1(3)$, where
	\begin{eqnarray*}
		A_1=\left(\begin{array}{ccc}
			0 & 0 & 0  \\
			0 & 0 &  1\\
			0 & -1 &  0
		\end{array}\right),\,\,\,A_2= \left(\begin{array}{ccc}
			0 & 0 & 1  \\
			0 & 0 &  1\\
			1 & -1 &  0
		\end{array}\right)\,\,\,\, \text{e} \,\,\,\,\, 
		A_3=\left(\begin{array}{ccc}
			0 & 1 & 0  \\
			1 & 0 &  0\\
			0 & 0 &  0
		\end{array}\right).
	\end{eqnarray*}
	Then $M^{\prime}(0)=c_{1} A_1+c_{2}A_2+c_{3}A_3$, for real numbers $c_i$.
	
	
	By  simple computations, we can prove that $\displaystyle \spn{M^{\prime}(0)X(s),N(s)}=\spn{T(s),v}$, where $X(s)\times N(s)=-T(s)$ and $v=(c_1+c_2,c_2,-c_3)$. Therefore,
	\begin{equation*}
	k(s)=\spn{T(s),v}.
	\end{equation*}
	
	Conversely, let $X(s)$ be a curve in $\mathbb{H}^{2}\subset \mathbb{R}^{3}_1$ parametrized by arc length $s$, such that $\langle T(s),v\rangle=k(s)$ for a vector $v \in \mathbb{R}^{3}_1\setminus \{0\}$. Without loss of generality, up to isometries of $\mathbb{H}^{2}$, we can consider $v$ to be a multiple of $w_1=(1,0,0)$ if $v$ is a timelike vector, a multiple of $w_2=(1,1,0)$ if $v$ is a lightlike vector and a multiple of $w_3=(0,0,-1)$ if $v$ is a spacelike vector. Thus, depending on the type of the vector $v$,  we can write the curvature as  $k_i(s)=\langle T(s),v_i\rangle$ where $v_i=aw_i$, $a>0$ and $i=1,2,3$. 
	
	Now, we define the evolution of $X$ in $\mathbb{H}^{2}$ to be  $\hat{X}_{i}(s,t)=M_{i}(t)X(s)$, where 	
	\begin{eqnarray} 
	M_{1}(t)&:=&\left(\begin{array}{lll}
	1 & 0                    &                   0  \\
	0 & cos(\varphi_{1}(t)) &  sen(\varphi_{1}(t))\\
	0 & -sen(\varphi_{1}(t)) &  cos(\varphi_{1}(t))
	\end{array}\right),\,\,\,\nonumber
	M_{2}(t):=\left(\begin{array}{ccc}
	1+\frac{(\varphi_{2}(t))^{2}}{2} & -\frac{(\varphi_{2}(t))^{2}}{2}  &   \varphi_{2}(t)  \\
	\frac{(\varphi_{2}(t))^{2}}{2} & 1-\frac{(\varphi_{2}(t))^{2}}{2}  &   \varphi_{2}(t)  \\
	\varphi_{2}(t)                    & -\varphi_{2}(t)                     &   1 
	\end{array}\right),\label{1.1}\\
	M_{3}(t)&:=&\left(\begin{array}{lll}
	cosh(\varphi_{3}(t)) & senh(\varphi_{3}(t))  &   0  \\
	senh(\varphi_{3}(t)) & cosh(\varphi_{3}(t))  &   0  \\
	0                    & 0                     &   1 
	\end{array}\right)\nonumber,
	\end{eqnarray}
	and $\varphi_{i}(t)=at$ for each $i=1,2,3.$

	
	A straightforward computation shows that 
	$$\displaystyle \left\langle M^{\prime}_{i}(t)X(s),M_{i}(t)N(s)\right\rangle=-\varphi^{\prime}_i(t)\langle X(s)\times N(s),w_{i}\rangle =-\varphi^{\prime}_i(t)\langle-T(s),w_{i}\rangle.$$
	Thus, 
	
	\begin{eqnarray*}
		\left\langle \frac{\partial}{\partial t}\hat{X}_{i}(s,t),\hat{N}_{i}(s,t)\right\rangle&=&\left\langle M^{\prime}_{i}(t)X(s),M_{i}(t)N(s)\right\rangle\\
		&=&-\varphi^{\prime}_i(t)\langle-T(s),w_{i}\rangle\;=\;
		\langle T(s),v_{i}\rangle\\
		&=&k_{i}(s)\; =\; 
		\hat{k}_{i}(s,t),
	\end{eqnarray*}
	where the last equality follows from the fact that isometries preserve geodesic curvature. 
	Therefore,  $X$ is a soliton solution to the CSF. 
	
	\mbox{} \hfill  $\square$ 
	
It follows from Theorem \ref{c2t1} that the study of the solitons solutions to the CSF on the 2-dimensional hyperbolic space is reduced to describing the curves that satisfy Equation \eqref{eq2} for some vector $v \in \mathbb{R}^{3}_1\setminus \{0\}$. Up to isometries of $\mathbb{H}^{2}$ we consider $v$ as being $v_{i}=ae_{i}$, where $a\in \R^{+}$, $e_1=(-1,0,0)$ if $v$ is a timelike vector,  $e_1=(-1,1,0)$ if $v$ is a lightlike vector and $e_1=(0,0,1)$ if $v$ is a spacelike vector. Our next result characterizes \eqref{eq2} in terms of a system of differential equations.

\begin{proposition}\label{p1}
	Let $\displaystyle X: I \rightarrow \mathbb{H}^{2}$ be a regular curve parametrized by arc length $s$. Consider the vectors 
	 \begin{equation}\label{e}
	e_1=(-1,0,0),\qquad e_2=(-1,1,0) \qquad  e_3=(0,0,1).	
	\end{equation}
	For each $i\in \{1,2,3\}$, define the functions 
	$$
	\alpha_{i}(s)=\langle X(s),e_{i}\rangle, \qquad \tau_{i}(s)=\langle T(s),e_{i}\rangle
	  \qquad \eta_{i}(s)=\langle N(s),e_{i}\rangle, 
	  $$ where $T$ and $N$ are the unit vector fields tangent and normal  to $X$, respectively. For a fixed $a>0$,  
	  $$k_{i}(s)=a\tau_{i}(s)$$
	is satisfied,   for all $ s\in I$ if, and only if, the functions $\alpha_{i}(s)$, $\tau_{i}(s)$ and $\eta_{i}(s)$ satisfy the system 
	\begin{equation}\label{eq3}
	\left\{\begin{array}{lll}
	\alpha^{\prime}_{i}(s)=\tau_{i}(s),\\
	\tau^{\prime}_{i}(s)=a\tau_{i}(s)\eta_{i}(s)+\alpha_{i}(s),\\
	\eta^{\prime}_{i}(s)=-a\tau^{2}_{i}(s),
	\end{array}\right. 
	\end{equation}
	with initial condition $(\alpha_{i}(0),\tau_{i}(0),\eta_{i}(0))$ satisfying 
	\begin{equation} \label{eq4}
	-\alpha^{2}_{i}(0)+\tau^{2}_{i}(0)+\eta^{2}_{i}(0)=
	\left\{\begin{array}{cll}
	-1, &\text{if} & i=1,\\
	0,  & \text{if} & i=2,\\
	1,  & \text{if} &  i=3.
	\end{array}\right.
	\end{equation}

For such functions, the expression $-\alpha^{2}_{i}(s)+\tau^{2}_{i}(s)+\eta^{2}_{i}(s)$ is equal to the right hand side  of \eqref{eq4}, for all $s \in I$. Moreover, $\eta_{i}(s)$ is a decreasing function.
\end{proposition}

\begin{proof}
	The vector fields  $X$, $T$ and $N$ satisfy the following system of equations 	
	\begin{equation} \label{eq5a}
	\left\{\begin{array}{lll}
	X^{\prime}(s)=T(s),\\
	T^{\prime}(s)=k(s)N(s)+X(s),\\
	N^{\prime}(s)=-k(s)T(s).
	\end{array}\right.
	\end{equation}
	Taking the inner product with $e_i$,  we get that $\alpha_{i}(s)$, $\tau_{i}(s)$ and $\eta_{i}(s)$ satisfy the system of equations
	
	\begin{equation}
	\left\{\begin{array}{lll}\label{eq5}
	\alpha^{\prime}_{i}(s)=\tau_{i}(s),\\
	\tau^{\prime}_{i}(s)=k_{i}(s)\eta_{i}(s)+\alpha_{i}(s),\\
	\eta^{\prime}_{i}(s)=-k_{i}(s)\tau_{i}(s),
	\end{array}\right.
	\end{equation}
	
	Suppose that $k_{i}(s)=a\tau_{i}(s)$ for all $s \in I$. Then substituting into  \eqref{eq5}, we obtain \eqref{eq3}. Note that, 
	\begin{equation*}
	e_i=\alpha_i(s)X(s)+\tau_i(s)T(s)+\eta_iN(s).
	\end{equation*}
	Therefore, $\langle e_i,e_i\rangle= -\alpha^{2}_{i}(s)+\tau^{2}_{i}(s)+\eta^{2}_{i}(s)$ is constant for all $s \in I$. In particular, for $s=0$, we obtain \eqref{eq5}. Moreover, it follows from the third equation of the system  \eqref{eq3} that the function $\eta_{i}(s)$ is  decreasing.  	
	
	Conversely, suppose that the functions $\alpha_{i}(s)$, $\tau_{i}(s)$ and $\eta_{i}(s)$ satisfy \eqref{eq3} and \eqref{eq4} for each $i \in \{1,2,3\}$. Since \eqref{eq5} holds, we have
	\begin{eqnarray*}
		\left\{\begin{array}{l}
			a\tau_{i}(s)\eta_{i}(s)+\alpha_{i}(s)=k_{i}(s)\eta_{i}(s)+\alpha_{i}(s),\\
			-(a\tau_{i}(s))\tau_{i}(s)=-k_{i}(s)\tau_{i}(s),
		\end{array}	\right.
	\end{eqnarray*}
	i.e.,
	\begin{eqnarray*}
		\left\{\begin{array}{l}
			\left[a\tau_{i}(s)-k_{i}(s)\right]\eta_{i}(s)=0,\\
			\left[a\tau_{i}(s)-k_{i}(s)\right]\tau_{i}(s)=0,
		\end{array}	\right.
	\end{eqnarray*}
	for all $s \in I$. For eaxch $i$, in order to conclude that $k_i(s)=a\tau_i(s)$, for all $s$, we will assume that $k_i(s)\neq a\tau_i(s)$  at some point $s_0$. Then this will occur on some interval $J\subset I$ around $s_0$. Hence  $\eta_{i}(s)=\tau_{i}(s)=0 $ for $s\in J$. Therefore, $e_i$ will be orthogonal to $T(s)$ and $N(s)$  for all $s \in J$. Thus, $e_i$ will be parallel to $X(s)$ for all $s \in J$. But $e_i$ is a constant vector for each $i$, so this can only happen at some isolated points of a curve $X$ in $\mathbb{H}^{2}$, which is a contradiction. Therefore, $k_i(s)=a\tau_i(s)$ for all $s\in I$ and for each $i \in \{1,2,3\}$.
\end{proof}

Our next proposition shows  how a  solution of the system \eqref{eq3}, with initial conditions satisfying \eqref{eq4},  is related to a soliton solution to the CSF.

\begin{proposition}\label{p2}
	Given a solution $(\alpha(s),\tau(s),\eta(s))$ to the system \eqref{eq3} on some interval $J$ with fixed $a>0$ and initial conditions $(\alpha(0),\tau(0), \eta(0))$ satisfying $-\alpha^{2}(0)+\tau^{2}(0)+\eta^{2}(0)=-1$ (resp. $0$ and $1$), there exists a smooth curve $\displaystyle X: I \rightarrow \mathbb{H}^{2}$ parametrized by arc length $s$, such that its tangent and normal unit vector fields $T$ and $N$ satisfy
	\begin{equation}\label{eq6}
	\alpha(s)=\langle X(s),e\rangle,\hspace{.5 cm}\tau(s)=\langle T(s),e\rangle \hspace{.5 cm} \text{and} \hspace{.5 cm} \eta(s)=\langle N(s),e\rangle, 
	\end{equation} 
	where $e=(-1,0,0)$ (resp. $e=(-1,1,0)$ and $e=(0,0,1)$).
	
\end{proposition}
\begin{proof}
	Define $k(s)=a\tau(s)$. Thus, up to isometries of $\mathbb{H}^{2}$, there exists an unique curve $\displaystyle X: I \rightarrow \mathbb{H}^{2}$, whose curvature is $k(s)$ i.e. $X(s)$ and its tangent and normal unit vector fields $T(s)$ and $N(s)$ satisfy the system \eqref{eq5a}. The curve $X(s)$ is uniquely determined by the initial conditions $X(0)$, $T(0)$ and $N(0)$, that can be chosen such that
	$$-\alpha(0)X(0)+\tau(0)T(0)+\eta(0)N(0)=e,$$
	where $e=(-1,0,0)$ (resp. $e=(-1,1,0)$ and $e=(0,0,1)$). A straightforward computations shows that \eqref{eq4} and \eqref{eq5a} imply
	
	\begin{eqnarray*}
		\frac{d}{ds}\left[-\alpha(s)X(s)+\tau(s)T(s)+\eta(s)N(s)\right]=0.
	\end{eqnarray*}
	
	Therefore, \eqref{eq6} is satisfied.	
\end{proof}

\begin{remark}\label{ob1} \rm
	Let $\displaystyle X: I \rightarrow \mathbb{H}^{2}$ be a regular curve parametrized by arc length $s$ given by $X(s)=(x_1(s),x_2(s), x_3(s))$. The function $\alpha(s)$ defined by \eqref{eq6} has the following geometric interpretation.
	\begin{itemize}
		\item If $e=(-1,0,0)$ (timelike vector), then $\alpha(s)=x_1(s)>0$ for all $s \in I$. Moreover, $\alpha(s)$ is the height function with respect to the vector $(1,0,0).$
		\item If $e=(-1,1,0)$ (lightlike vector), then $\alpha(s)=x_1(s)+x_2(s)>0$ for all $s \in I$. Moreover, $\alpha(s)$ is the height function with respect to the vector $(1,1,0).$
		\item If $e=(0,0,1)$ (spacelike vector), then $\alpha(s)=x_3(s)$ for all $s \in I$. Moreover, $\alpha(s)$ is the height function (with sign) with respect to the vector $(0,0,1).$
	\end{itemize}

Figure \ref{H} (resp. \ref{C} and \ref{S}) provides a geometric illustration of the function $\alpha(s)$ as a height function with respect to vector $(-1,0,0)$ (resp. $e=(-1,1,0)$ and $e=(0,0,1)$.)
\end{remark}
\begin{figure}[!htb]
	\centering
	\subfloat[]{
		\includegraphics[height=2 cm]{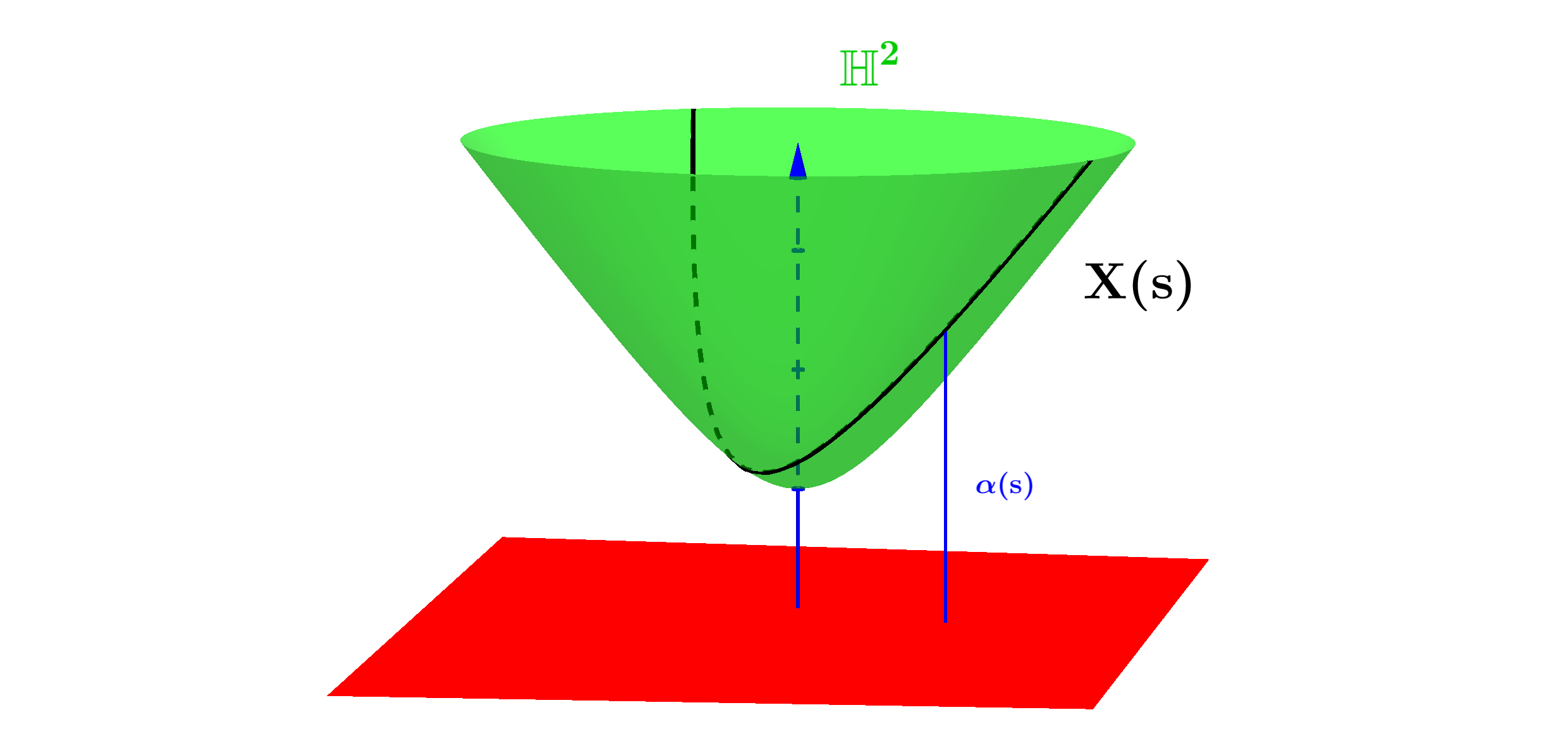}
		\label{H}
	}
	\quad 
	\subfloat[]{
		\includegraphics[height=2 cm]{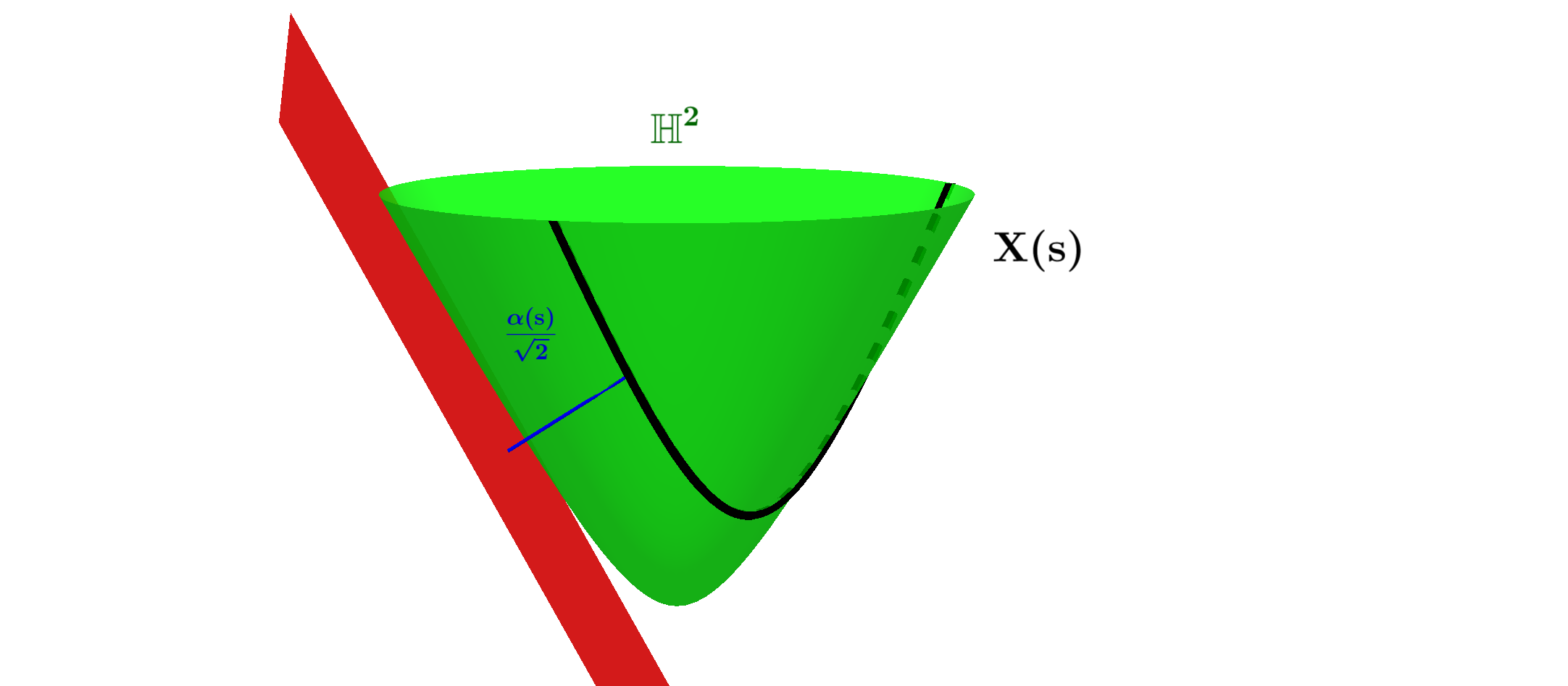}
		\label{C}
	}
	\quad 
	\subfloat[]{
		\includegraphics[height= 2 cm]{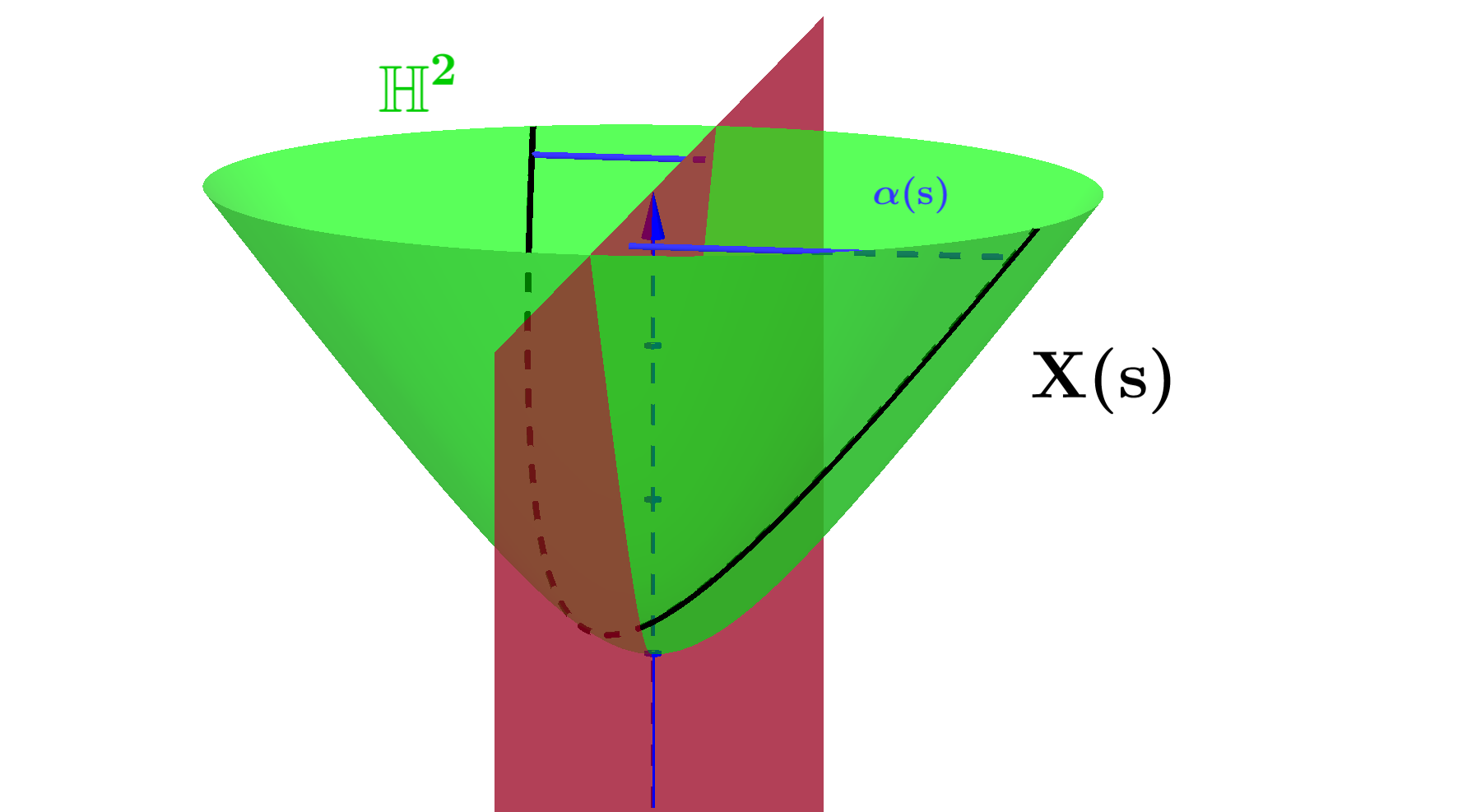}
		\label{S}
	}
	\caption{Geometric interpretation  of the functions $\alpha(s)$.}
	\label{fig4}
\end{figure}

As we have seen in Propositions \ref{p1}, \ref{p2} and Remark \ref{ob1}, the investigation  of the  soliton solutions to the CSF on the 2-dimensional hyperbolic space is equivalent to studying  the  solutions \mbox{$\psi(s)=(\alpha(s),\tau(s), \eta(s))$} of the system 
\begin{equation}\label{s1}
\left\{\begin{array}{lll}
\alpha^{\prime}(s)=\tau(s)\\
\tau^{\prime}(s)=a\tau(s)\eta(s)+\alpha(s)\\
\eta^{\prime}(s)=-a\tau^{2}(s),
\end{array}\right.
\end{equation}
for each constant $a>0$ and initial condition $\displaystyle \psi(0) \in H \cup C \cup S \subset \mathbb{R}^{3},$ where 
\begin{equation}\label{c} \begin{array}{l}
H:=\{(\alpha,\tau,\eta) \in \mathbb{R}^{3}: -\alpha^2+\tau^{2}+\eta^2=-1, \alpha>0\},\\	
C:=\{(\alpha,\tau,\eta) \in \mathbb{R}^{3}\setminus \{0\}: -\alpha^2+\tau^{2}+\eta^2=0, \alpha>0\},\\
S:=\{(\alpha,\tau,\eta) \in \mathbb{R}^{3}: -\alpha^2+\tau^{2}+\eta^2=1\}.	
\end{array}
\end{equation}
These are disjoint sets and if the initial condition 
$\psi(0)\in H$ (resp. $C$ or $S$) then the solution $\psi(s)$ defined on the maximal interval $I$ will  be contained in $ H$ (resp. $C$ or $S$) for all $s\in I$. 

From now on, using \eqref{s1}, we will prove a series of lemmas that will provide the proof of the main result (Theorem \ref{c2t2}). Namely, we will prove that for any initial condition, the solutions $\psi(s)$ of \eqref{s1} and hence the associated soliton solutions to  the hyperbolic space are defined on the whole $\R$. Moreover, we will analize the behaviour of the curvature function of the solitons at each end.  

In the first lemma we will study the solution of \eqref{s1} such that the function $\tau(s)$ is constant. As we will see such solutions (that will be called trivial) only exit on $S$. 

\begin{lemma} \label{l1}
	Let $\psi(s)=(\alpha(s),\tau(s),\eta(s))$ be a non null solution of \eqref{s1} defined on the maximal interval $I=(\omega_{-},\omega_{+})$, $a>0$ and initial condition $\psi(0) \in H\cup C \cup S$, where $H$, $C$ and $S$ are given by \eqref{c}. Then 
	the function $\tau(s)=b$, $s \in I$, where $b$ is a real constant if, and only if, $b \in \{-1,0,1\}$, $I=\mathbb{R}$ and $\psi(s) \in S$ for all $s \in \mathbb{R}.$ Moreover,
	\begin{itemize}
		\item[i)] If $b=0$, then $\psi(s)=(0,0,\pm 1)$ are singular solutions of \eqref{s1} in $S$. 
		\item[ii)] If $b^{2}=1$, then $a=1$ and $\psi(s)=(\pm s+\alpha(0),\,\pm 1,\,-s\pm \alpha(0)).$
	\end{itemize}
\end{lemma}
\begin{proof} 
	If $b=0$, it follows from \eqref{s1} that $\alpha(s)=0$ for all $s \in I$. Using the equation $-\alpha^{2}(s)+\tau^{2}(s)+\eta^{2}(s)=\gamma$, where $\gamma \in \{-1,0,1\}$, we obtain $\eta^{2}(s)=1$ for all $s \in I$. Hence, $\psi(s)=(0,0,\pm 1)$, $\forall\; s \in \mathbb{R}$ are singular solutions of \eqref{s1} in $S$.
	
	If $\psi(s)$ is not a singular solution, then $b\neq 0$ and it follows from \eqref{s1} that $$a\,b\,\eta(s)=-\alpha(s)\hspace{1 cm} \text{and} \hspace{1 cm} \eta(s)=-ab^{2}s+\eta(0),$$ for all $s \in \mathbb{R}$. Using the relation $-\alpha^{2}(s)+\tau^{2}(s)+\eta^{2}(s)=\gamma$, where $\gamma \in \{-1,0,1\}$, we conclude that 
	$$
	[-a^{2}\,b^{2}+1]\eta^{2}(s)=\gamma -b^{2},
	$$ 
and hence the function $[-a^{2}b^{2}+1]\alpha^{2}(s)$ is also constant. Since $\psi(s)$ is not a singular solution, it follows that  $\displaystyle a^{2}b^{2}=1$  and $b^2=\gamma$.
Therefore, $\gamma=1$, $b=\pm 1$, $a=1$ and  $\alpha(s)=\mp\eta(s)$,  for all $s \in \mathbb{R}$.  This concludes the proof.   
 \end{proof}

It follows from Lemma \ref{l1} that when $\tau(s)$ is a constant function then the functions $\alpha(s)$ and $\eta(s)$ are linear in $s$  and its corresponding soliton solutions to the CSF are curves of constant curvature i.e. geodesics  when $k(s)=\tau(s)=0$ or planar curves  with curvature $k(s)=\tau(s)=\pm 1$. 

In this context, we define a {\it trivial solution}  $\psi(s)=(\alpha(s),\tau(s),\eta(s))$ of \eqref{s1},  when $\tau(s)$ is a constant function. From now on, we will study only non trivial solutions of \eqref{s1}. It follows from Lemma \ref{l1} that there are no trivial solutions of  \eqref{s1} in $H\cup C$. 

On our next lemmas we will study the solutions  $\psi(s)$ of \eqref{s1} contained in $H\cup C$ and those contained in $S$ separately.  

\begin{lemma}\label{l2}
	Let $\psi(s)=(\alpha(s),\tau(s),\eta(s))$ be a  solution of \eqref{s1} defined on the maximal interval $I=(\omega_{-},\omega_{+})$, $a>0$ and initial condition $\psi(0) \in H\cup C $, where $H$ and $C$ are given by \eqref{c}.
	\begin{itemize}
		\item[i)] If $\alpha(s)$ has a critical point then it is a global minimum point of $\alpha$. Moreover, there exists always $\overline{s} \in I$ such that $\alpha(s)$ is strictly monotone on the intervals $(\omega_{-},\overline{s}]$ and $[\overline{s}, \omega_{+})$. 
		\item[ii)] If $s_0$ is a critical point of $\tau(s)$, then $a^{2}\tau^{2}(s_0)>1$ and $s_0$ is a local minimum (resp. maximum) point of $\tau(s)$ if, and only if, $\tau(s_0)<0$ (resp. $\tau(s_0)>0$).
	\end{itemize}
\end{lemma}

\begin{proof}	
	{\it i)}  Let $s_0$ be a critical point of $\alpha(s)$. Note that $\alpha(s)>0$ for all $s \in I$ whenever $\psi(0) \in H\cup C $.  Taking the second derivative of $\alpha(s)$ and using \eqref{s1} at $s=s_0$, we have
	\begin{equation*}
	\alpha^{\prime \prime}(s_{0})=\tau^{\prime}(s_{0})=a\tau(s_{0})\eta(s_{0})+\alpha(s_{0})=\alpha(s_{0})>0.
	\end{equation*}
	Hence, $s_0$ is a global minimum point of $\alpha(s)$. Therefore, $\alpha(s)$ has at  most one critical point. If  there are no critical points then $\alpha(s)$ is trictly monotone on $I$. 
	
	
	{\it ii)}  Let $s_0$ be  a critical point of $\tau(s)$. Then $\tau^{\prime}(s_{0})=a\eta(s_{0})\tau(s_{0})+\alpha(s_{0})=0$ and $\eta(s_{0})\tau(s_{0})\neq 0$ because $\alpha(s)>0$ for all $s \in I$. Since $\psi(0)\in H\cup C$, it follows that  $-\alpha^{2}(s)+\tau^{2}(s)+\eta^{2}(s)=\delta\leq0$ for all $s \in I$, where $\delta \in \{-1,0\}$. Thus,
	\begin{equation*}
	\delta=-a^{2}\tau^{2}(s_{0})\eta^{2}(s_{0})+\tau^{2}(s_{0})+\eta^{2}(s_{0})=\eta^{2}(s_{0})[-a^{2}\tau^{2}(s_0)+1]+\tau^{2}(s_{0})\leq 0.
	\end{equation*}
	Hence, $-a^{2}\tau^{2}(s_0)+1<0$. Taking the second derivative of $\tau(s)$ and using \eqref{s1} at $s=s_0$, we have
	\begin{equation}\label{t}
		\tau^{\prime \prime}(s_{0})=a\tau^{\prime}(s_{0})\eta(s_{0})+a\tau(s_{0})\eta^{\prime}(s_{0})+\alpha^{\prime}(s_{0})
		=\tau(s_{0})\left[-a^{2}\tau^{2}(s_{0})+1\right].
	\end{equation}
	This concludes the proof of {\it ii)}.
\end{proof} 

\begin{lemma}\label{l3}
	Let $\psi(s)=(\alpha(s),\tau(s),\eta(s))$ be a  solution of \eqref{s1} defined on the maximal interval $I=(\omega_{-},\omega_{+})$, $a>0$ and initial condition $\psi(0) \in C $, where $C$ is given by \eqref{c}.
	\begin{itemize}
		\item[i)] If $\tau(s)>0$ in $ I$, then $\alpha(s)$ is strictly increasing in $I$, $\tau(s)$ is bounded and it has at  most one critical point in $I$. Moreover,  $\omega_{-}=-\infty$, $\displaystyle \lim_{s \to -\infty} \psi(s)=(0,0,0)$ and $\displaystyle \lim_{s \to \omega_{+}} -\eta(s)=\lim_{s \to \omega_{+}} \alpha(s)=+\infty.$ 
		\item[ii)] If $\tau(s)<0$ in $I$, then $\alpha(s)$ is strictly decreasing in $I$, $\tau(s)$ is bounded and it has at  most one critical point in $I$. Moreover,  $\omega_{+}=+\infty$, $\displaystyle \lim_{s \to +\infty} \psi(s)=(0,0,0)$ and $\displaystyle \lim_{s \to \omega_{-}} \eta(s)=\lim_{s \to \omega_{-}} \alpha(s)=+\infty.$
	\end{itemize}
\end{lemma}
\begin{proof} {\it i)}  If $\tau(s)>0$, then it follows from Lemma \ref{l2} that $\tau(s)$ has only local maximum points i.e. $\tau(s)$ has at  most one critical point. The positive function $\alpha(s)$ is bounded and there exists  $\overline{s} \in I$ such that $\alpha$ is  strictly increasing on $(\omega_{-},\overline{s})$. Thus, it follows from equation $\alpha^{2}(s)=\tau^{2}(s)+\eta^{2}(s)$ that we can take $\overline{s}$ such that $\tau(s)$ and $\eta(s)$ are bounded and monotone on $(\omega_{-},\overline{s})$. The interval $I$ is  maximal, hence $\omega_{-}=-\infty.$ Since the limits $\displaystyle \lim_{s \to -\infty} \alpha(s)$, $\displaystyle \lim_{s \to -\infty} \alpha^{\prime}(s)=\lim_{s \to -\infty}\tau(s)$ and $\displaystyle \lim_{s \to -\infty} \tau^{\prime}(s)$ exist, we obtain 
that $\displaystyle \lim_{s \to -\infty} \tau(s)=0$ and $\displaystyle \lim_{s \to -\infty} \alpha(s)=\lim_{s \to -\infty} \eta(s)=0$. Using the fact that the function $\eta(s)$ is  decreasing, we get that $\eta(s)<0$ for all $s \in I.$
	
	We claim that $\alpha(s)$ is ununbounded on $(\overline{s},\omega_{+})$. In fact, assume by contradiction that the strictly increasing function $\alpha(s)$ is bounded on $(\overline{s},\omega_{+})$. Thus, it follows from equation $\alpha^{2}(s)=\tau^{2}(s)+\eta^{2}(s)$ that we can take $\overline{s}$ such that $\tau(s)$ and $\eta(s)$ are bounded and monotone on $(\overline{s},\omega_{+})$. Hence, there exists $p \in C $ such that $\displaystyle \lim_{s \to \omega_{+}}(\alpha(s),\tau(s),\alpha(s))=p$ and $p$ is a singular (trivial) solution in $C$, which contradicts  Lemma \ref{l1}. Therefore, $\displaystyle \lim_{s\to \omega_{+}}\alpha(s)=+\infty.$
	
	Now, assume by contradiction that the strictly decreasing and negative function $\eta(s)$ is bounded on $(\overline{s},\omega_{+})$. Since $\tau^{2}(s)+\eta^{2}(s)=\alpha^{2}(s)$, it follows that the function $\tau(s)$ is unbounded and positive on $(\overline{s},\omega_{+})$,   because we showed that $\alpha(s)$ is unbounded on $(\overline{s},\omega_{+})$. Thus, we can choose $\overline{s}$ such that $2\tau(s)<a\tau^{2}(s)$ for all $s>\overline{s}$. Using the equations of \eqref{s1}, we obtain
	\begin{equation}\label{ps2}
	2\alpha(s)-2\alpha(\overline{s})=2\int_{\overline{s}}^{s}\tau(s)ds<\int_{\overline{s}}^{s}a\tau^{2}(s)ds=-\eta(s)+\eta(\overline{s}).
	\end{equation}
	Hence	
	\begin{equation*}
	2\alpha(s)<-\eta(s)+\eta(\overline{s})+2\alpha(\overline{s}),
	\end{equation*}
	for each $s \in (\overline{s},\omega_{+})$.  But this contradicts the fact that  $\alpha(s)$ is unbounded. Therefore, $\eta(s)$ is unbounded and  $\displaystyle \lim_{s \to \omega_{+}} \eta(s)=-\infty$.
	
	Finally, if $\tau(s)$ is unbounded on $(\overline{s},\omega_{+})$, $\overline{s} \in I$, then we can choose again $\overline{s}$ such that $2\tau(s)<a\tau^{2}(s)$ for all $s>\overline{s}$. Using \eqref{ps2}, we obtain
	\begin{equation*}
	\alpha(s)+\eta(s)< 2\alpha(\overline{s})-\eta(\overline{s})-\alpha(s),
	\end{equation*}
	this is a contradiction, because $\alpha^{2}(s)=\tau^{2}(s)+\eta^{2}(s)$ i.e. $\alpha(s)>-\eta(s)$ and $\displaystyle \lim_{s \to \omega_{+}} \alpha(s)=+\infty$.
	
	\it ii) \rm This proof is analogous to the proof of item {\it i)}. 
\end{proof}

\begin{lemma}\label{l4}
	Let $\psi(s)=(\alpha(s),\tau(s),\eta(s))$ be a  solution of \eqref{s1} defined on the maximal interval $I=(\omega_{-},\omega_{+})$, $a>0$ and initial condition $\psi(0) \in H $, where $H$ is given by \eqref{c}. Then there exists a unique  $s_0$ such that $\alpha^{\prime}(s_0)=\tau(s_0)=0$.
\end{lemma}
\begin{proof}
	Assume by contradiction that such  an $s_0$ does not exist. Then, either $\tau(s)<0$ or $\tau(s)>0$ for all $s \in I$.  If  $\tau(s)<0$, then $\alpha(s)$ is strictly decreasing. Taking $\overline{s} \in I$ we have $1\leq \alpha(s)\leq \alpha(\overline{s})$ for all $s>\overline{s}$. Since  $-\alpha^2(s)+\tau^{2}(s)+\eta^{2}(s)=-1$, then $\tau^{2}(s)+\eta^{2}(s)<\alpha^2(s) <\alpha^2(\overline{s})$ for all $s>\overline{s}$. Hence, the functions $\alpha(s), \eta(s)$ are bounded and monotone in $[\overline{s},\omega_{+})$, and $\displaystyle \lim_{s \to \omega_{+}}\tau(s)$ exists, because $\tau^{2}(s)=-1+\alpha^{2}(s)-\eta^{2}(s)$. Thus, there exists a point $p \in H$ such that $\displaystyle \lim_{s \to \omega_{+}}(\alpha(s),\tau(s),\alpha(s))=p$. Therefore, $\displaystyle \omega_{+}=+\infty$ and $p$ is a singular (trivial) solution of \eqref{s1}, which contradicts  Lemma \ref{l1}.	
	In a similar way, we can prove that $\tau(s)<0$ for all $s \in I$ cannot occur.
	
	Therefore, there is an $s_0\in I$ such that $\alpha^{\prime}(s_0)=\tau(s_0)=0$. It follows from  Lemma \ref{l2}, that $s_0$ is a global minimum of the function $\alpha(s)$. Hence, $s_0$ is unique.
\end{proof}

In the next three lemmas, we will suppose that $\psi(0) \in H\cup C$ and that $\alpha(s)$ has  only one critical point. Note that, this hypothesis only excludes the case presented in  Lemma \ref{l3} because it follows from Lemma \ref{l1} that $\alpha(s)$ has at  most one critical point when $\psi(0)\in H\cup C$ and  Lemma \ref{l4} shows that $\alpha(s)$ has a unique critical point, when $\psi(0)\in H$.

\begin{lemma}\label{l5}
	Let $\psi(s)=(\alpha(s),\tau(s),\eta(s))$ be a  solution of \eqref{s1} defined on the maximal interval $I=(\omega_{-},\omega_{+})$, $a>0$ and initial condition $\psi(0) \in H\cup C $, where $H$ and $C$ are given by \eqref{c}.
	If $\alpha(s)$ has one critical point, then $\displaystyle \lim_{s \to \omega_{-}}\alpha(s)=\lim_{s \to \omega_{+}}\alpha(s)=\infty$. 
\end{lemma}
\begin{proof} Let $s_0$ be the global minimum point of $\alpha(s)$. Thus, $\alpha(s)$ is monotone on the intervals $(\omega_{-},s_{0}]$ and $[s_{0}, \omega_{+})$. Assume by contradiction that $\alpha(s)$ is bounded on the intervals $(\omega_{-},s_{0}]$ and $[s_{0}, \omega_{+})$. Since the functions $\alpha(s)$, $\tau(s)$ and $\eta(s)$ satisfy $\tau^{2}(s)+\eta^{2}(s)=\delta + \alpha^2(s)\leq\alpha^2(s)$, where $\delta \in \{-1,0\}$, then the functions $\alpha(s)$ and $\eta(s)$ are bounded and monotone on $(\omega_{-},s_{0}]$ and $[s_{0}, \omega_{+})$. The limits $\displaystyle \lim_{s \to \omega_{-}}\tau(s)$ and $\displaystyle \lim_{s \to \omega_{+}}\tau(s)$ exist, because $\tau^{2}(s)=\delta+\alpha^{2}(s)-\eta^{2}(s)$. Thus, there are points $p_1$ and $p_2 \in H\cup C$ such that $\displaystyle \lim_{s \to \omega_{-}}(\alpha(s),\tau(s),\alpha(s))=p_1$ and $\displaystyle \lim_{s \to \omega_{+}}(\alpha(s),\tau(s),\alpha(s))=p_2$. Hence, $\omega_{-}=-\infty$, $\omega_{+}=+\infty$ and $\{p_1,p_2\}$ is a set of singular (trivial) solutions of \eqref{s1}, which contradicts  Lemma \ref{l1}. 
	
	Therefore, we conclude that the function $\alpha(s)$ is unbounded on the intervals $(\omega_{-},s_{0}]$ and $[s_{0}, \omega_{+})$, and moreover $\displaystyle \lim_{s \to \omega_{-}}\alpha(s)=\lim_{s \to \omega_{+}}\alpha(s)=\infty$.
\end{proof}

\begin{lemma}\label{l6}
	Let $\psi(s)=(\alpha(s),\tau(s),\eta(s))$ be a solution of \eqref{s1} defined on the maximal interval $I=(\omega_{-},\omega_{+})$, $a>0$ and initial condition $\psi(0) \in H\cup C $, where $H$ and $C$ are given by \eqref{c}.
	If $\alpha(s)$ has one critical point, then $\displaystyle \lim_{s \to \omega_{-}}\eta(s)=\infty$ and $\displaystyle \lim_{s \to \omega_{+}}\eta(s)=-\infty$.
\end{lemma}

\begin{proof}
	Let $s_{0}$ be the global minimum point of $\alpha(s)$. Then $\tau(s)<0$ for all $s<s_{0}$ and $\tau(s)>0$ for all $s>s_{0}$. Moreover, $\alpha(s)$ is unbounded and monotone on the intervals $(\omega_{-},s_{0}]$ and $[s_{0}, \omega_{+})$.
	
	Assume by contradiction that the function $\eta(s)$ is bounded on $(\omega_{-},s_{0}]$. Since  $\tau^{2}(s)+\eta^{2}(s)=\delta+\alpha^{2}(s)$, where $\delta \in \{-1,0\}$, then it follows from Lemma \ref{l4} that the function  $\tau(s)$ is unbounded and negative on $(\omega_{-},s_{0})$ i.e. there exists $s_{1} \in (\omega_{-},s_{0}]$ such that $a\tau(s)<-1$ and $ -a\tau^2(s)<\tau(s)$ for all $s \in (\omega_{-},s_{1}]$. Thus, using \eqref{s1} for each $s \in (\omega_{-},s_{1}]$, we obtain
	
	\begin{equation*}
	\alpha(s)-\alpha(s_1)=-\int_{s}^{s_{1}}\tau(s)ds<\int_{s}^{s_{1}}a\tau^{2}(s)ds=\eta(s)-\eta(s_{1}),
	\end{equation*}
	that is,
	\begin{equation*}
	\alpha(s)<\eta(s)-\eta(s_{1})+\alpha(s_{1}),
	\end{equation*}
	for each $s \in (\omega_{-},s_{1}]$ which contradicts  Lemma \ref{l5}. Therefore, $\eta(s)$ is unbounded on $(\omega_{-},s_{1}]$.
	
	In a similar way, we can prove that the function $\eta(s)$ is unbounded on $[s_{0}, \omega_{+})$.
	
	Since $\eta(s)$ is decreasing on $(\omega_{-},\omega_{+})$, it follows that $\displaystyle \lim_{s \to \omega_{-}}\eta(s)=\infty$ and $\displaystyle \lim_{s \to \omega_{+}}\eta(s)=-\infty$. 		
\end{proof}

\begin{lemma}\label{l7}
	Let $\psi(s)=(\alpha(s),\tau(s),\eta(s))$ be a  solution of \eqref{s1}, with $a>0$,  defined on the maximal interval $I=(\omega_{-},\omega_{+})$ and initial condition $\psi(0) \in H\cup C $, where $H$ and $C$ are given by \eqref{c}.
	If $\alpha(s)$ has one critical point, then the function $\displaystyle \tau(s)$ is unbounded on $I$ and it has only two critical points.
\end{lemma}

\begin{proof}
	Let $s_{0} \in I$ be the global minimum point of $\alpha(s)$. The arguments consist in studying  the existence and the properties of the critical points of $\tau(s)$.
	
	{\it Claim}. If $\tau(s)$ does not have any critical point on $I$, then $-1<a\tau(s)<0 $ on $(\omega_{-},s_{0})$ and $0< a\tau(s)<1$ on $(s_{0}, \omega_{+})$. In fact, suppose that $\tau^{\prime}(s)\neq0$ for all $s \in I$. At  $s_0$, $\tau(s_{0})=0$ and $\tau^{\prime}(s_{0})=\alpha(s_{0})>0$. Moreover, $\tau^{\prime}(s)=a\tau(s)\eta(s)+\alpha(s)>0$ i.e. $\tau(s)$ is an increasing function on $I$, $\displaystyle \lim_{s \to \omega_{-}}\tau(s)\neq 0$ and $\displaystyle \lim_{s \to \omega_{+}}\tau(s)\neq0$. It follows from Lemma \ref{l6} that $\displaystyle \lim_{s \to \omega_{-}}\eta(s)=+\infty$ and $\displaystyle \lim_{s \to \omega_{+}}\eta(s)=-\infty$. We also know that $\tau(s)$ is negative on $(\omega_{-},s_0)$ and positive on $(s_0,\omega_{+})$. Thus, there are $s_{1} \in (\omega_{-},s_{0})$ and $s_{2} \in (s_{0}, \omega_{+})$ such that $-\alpha(s)<a\tau(s)\eta(s)<0$ for all $s \in I\setminus[s_{1},s_{2}]$. Hence, $-\alpha^{2}(s)<-a^{2}\tau^{2}(s)\eta^{2}(s)$ for all $s \in I\setminus[s_{1},s_{2}]$ and
	
	\begin{equation*}
	\delta=-\alpha^{2}(s)+\tau^{2}(s)+\eta^{2}(s)<-a^{2}\tau^{2}(s)\eta^{2}(s)+\tau^{2}(s)+\eta^{2}(s),
	\end{equation*}
	i.e.,
	\begin{equation*}
	a^{2}\tau^{2}(s)\eta^{2}(s)<-\delta+\tau^{2}(s)+\eta^{2}(s). 	
	\end{equation*}
	Therefore, 
	\begin{equation*}
	1<\frac{-\delta}{a^{2}\eta^{2}(s)\tau^{2}(s)}+\frac{1}{a^2\eta^{2}(s)}+\frac{1}{a^{2}\tau^{2}(s)}
	\end{equation*}
	for all $s \in I\setminus[s_{1},s_{2}]$. 
	
	Taking the limit when $s \to \omega_{+}$ and $s \to \omega_{-}$, using the fact that $\tau(s)$ is increasing and  Lemma \ref{l6}, we obtain that 
	
	\begin{equation*}
	\lim_{s \to \omega_{+}} \frac{1}{a^{2}\tau^{2}(s)}>1\hspace{1cm}\text{and}\hspace{1cm}\lim_{s \to \omega_{-}} \frac{1}{a^{2}\tau^{2}(s)}>1,
	\end{equation*}
	i.e.,
	\begin{equation*}
	\lim_{s \to \omega_{+}} a^{2}\tau^{2}(s)<1\hspace{1cm}\text{and}\hspace{1cm}\lim_{s \to \omega_{-}} a^{2}\tau^{2}(s)<1.
	\end{equation*}
	Thus, using that $\tau(s)$ is increasing, we conclude that $-1<a\tau(s)<0$ on $(\omega_{-},s_{0})$ and $0<a\tau(s)<1$ on $(s_{0}, \omega_{+})$. This proves our Claim. 
	
	Still assuming that $\tau(s)$ does not have any critical point, we define the positive functions $f(s)=\alpha(s)+\eta(s)$ and $g(s)=\alpha(s)-\eta(s)$ (observe that $-\alpha^{2}(s)+\tau^{2}(s)+\eta^{2}(s)=\delta \leq0$). Taking the derivatives of $f$ and $g$ and using \eqref{s1}, we obtain $f^{\prime}(s)=\tau(s)[1-a\tau(s)]$ and $g^{\prime}(s)=\tau(s)[1+a\tau(s)]$. It follows from our Claim that the functions $f$ and $g$ are decreasing when $\tau(s)<0$ and they are increasing when $\tau(s)>0$ and $$\displaystyle 0<f(s)\cdot g(s)=-\delta+\tau^{2}(s)<\frac{-\delta a^{2}+1}{a^{2}}$$ for all $s \in I$, where $\delta \in \{-1,0\}$. Hence, the functions $f$ and $g$ are positive, monotone and bounded on the interval $(\omega_{-},s_{0})$. Similarly  one shows that the functions $f$ and $g$ are positive, monotone and bounded on $(s_{0}, \omega_{+})$. Thus, we conclude that there exist  $M_1,M_2 \in \mathbb{R}$ such that 
	\begin{equation*}
	\left\{\begin{array}{ll}
	\alpha(s)+\eta(s)\leq M_{1},\\  
	\alpha(s)-\eta(s)\leq M_{2},
	\end{array}\right.	
\qquad 	\forall\; s \in (\omega_{-},s_{0}) \cup (s_{0}, \omega_{+}).
	\end{equation*} 	 	
	Hence,
	\begin{equation*}
	2\alpha(s)\leq M_{1}+M_{2}\,\,\forall\,\, s\,\, \in\,\,(\omega_{-},s_{0})\,\, \text{and}\,\, 2\alpha(s)\leq M_{1}+M_{2}\,\,\forall\,\, s\,\, \in\,\, (s_{0}, \omega_{+}).
	\end{equation*}
	 
	These inequalities contradict  Lemma \ref{l5}. Hence, the function $\tau(s)$ has at least one critical point on each interval $(\omega_{-},s_{0})$ and $(s_{0}, \omega_{+})$. From item {\it ii)} of Lemma \ref{l2} we have that $\tau(s)$ has only one local minimum $ s_1 \in (\omega_{-},s_{0})$ and it has only one local maximum $ s_2 \in (s_{0}, \omega_{+})$ i.e. $\tau(s_{1})\leq\tau(s)\leq \tau(s_{2})$ for all $s \in I$ and $\tau(s)$ is bounded on the interval $I$. 	This concludes the proof.	
\end{proof}

\begin{lemma}\label{lb8}
	Let $\psi(s)=(\alpha(s),\tau(s),\eta(s))$ be a  solution of \eqref{s1} defined on the maximal interval $I=(\omega_{-},\omega_{+})$, $a>0$ and initial condition $\psi(0) \in H\cup C $, where $H$ and $C$ are given by \eqref{c}. Then $I=\R$.  
\end{lemma}
\begin{proof}
	 From Lemmas \ref{l3} and \ref{l7} we have that $\tau(s)$ is bounded. Let $M>0$ be such that $|\tau(s)|\leq M$, for all $s \in I$. Using \eqref{s1}, we obtain,
	\begin{equation}\label{34}
	|\alpha(s)-\alpha(s_{0})|=\left|\int_{s_{0}}^{s}\tau(s)ds\right|\leq M|s-s_{0}|
	\end{equation}
	for each $s \in I$. If $\alpha(s)$ has a global minimum point, then it follows from Lemma \ref{l5} that $\displaystyle \lim_{s \to \omega_{-}}\alpha(s)=\lim_{s \to \omega_{+}}\alpha(s)=\infty$. Hence, $I=\R$.		
	
	If $\alpha(s)$ does not have any critical point and $\tau(s)<0$ (resp. $\tau(s)>0$) for all $s \in I$, then it follows from Lemma \ref{l3} that $\omega_{+}=+\infty$ and $\displaystyle \lim_{s \to \omega_{-}}\alpha(s)=+\infty$ (resp. $\omega_{-}=-\infty$ and $\displaystyle \lim_{s \to \omega_{+}}\alpha(s)=+\infty$). Using \eqref{34}, we conclude that  $\omega_{-}=-\infty$ (resp. $\omega_{+}=+\infty$). Therefore, $I=\R$.
\end{proof}

We will now study the solutions of the system \eqref{s1}, with initial conditions on $S$. 
We will first classify the singular points of the system that are in the set  $S$.

\begin{lemma}\label{ld1}
	Let  $\Phi:S\rightarrow TS \subset \R^{3}$ be the differential vector field given by 
	$$
	\Phi(\alpha,\tau,\eta)=\left(\tau,a\tau\eta+\alpha,-a\tau^{2}\right),
	$$
	where $a>0$. Then $p=(0,0,1)$ and $-p=(0,0,-1)$ are the singular points of $\Phi$ and the eigenvalues of $d\Phi_{p}$ and $d\Phi_{-p}$ are given respectively by  
	\begin{equation}\label{eigen} 
	\lambda_{p}=\frac{a\pm \sqrt{a^2+4}}{2}, \hspace{.5 cm}   \lambda_{-p}=\frac{-a\pm \sqrt{a^2+4}}{2}.
	\end{equation}
			\end{lemma}

\begin{proof}
	Note that, if $\Phi(\alpha,\tau,\eta)=0$, then $\alpha=\tau=0$ and $\eta=\pm 1$. Hence, $p=(0,0,1)$ and $-p=(0,0,-1)$ are the singular points of $\Phi$. The tangent plane in each singular point is defined by $T_{\pm p}S\approx\{(\alpha,\tau,\eta) \in \R^{3}: \eta=0\}$. Thus,
	\begin{equation*}
	d\Phi=\left(\begin{array}{ccc}
	0                       & 1          & 0  \\
	1           & a\eta  & a\tau \\
	0                       & -2a\tau &  0
	\end{array}\right).
	\end{equation*}
	Hence $\lambda$ is an eigenvalue of $d\Phi_{\pm p}$ if there is a non null 
	vector $w=(w_1,w_2,0) \in T_{\pm p} S$ such that $d\Phi_{\pm p}(w)=\lambda w$
	 i.e.
	\begin{equation*}
	\left\{\begin{array}{cc}
	w_{2}=\lambda w_{1},\\
	w_{1}\pm aw_{2}=\lambda w_{2}.
	\end{array}\right. 
	\end{equation*}
	Hence, $\lambda$ satisfies $\lambda^{2} \mp a\lambda-1=0$, which gives \eqref{eigen}.
	
\end{proof}

In  Lema \ref{ld1}, we saw that $(0,0,\pm 1)$  are saddle points 
for the vector field  $\Phi$ on $S$, i.e., $\psi(s)=(0,0,\pm1)$, $s \in I$ 
are singular solutions of \eqref{s1}. If the functions $\alpha$ and $\tau$ are identically 
zero then the corresponding  curve  $X(s)$ is the intersection of the upper half hyperboloid 
with the plane going through the origin, orthogonal to  $(0,0,1)$.  Hence both singular solutions of the system correspond to the same curve. 

In order to study the non trivial solutions of the system 
\eqref{s1}, we consider the singular point   $p=(0,0,1)$ and  $\psi(s,q)$ a solution of the system with initial condition $q\in S$. Since the eigenvalues of the linearized system at the singular point are not zero, it follows that the local behavior of the system  
(\ref{s1}) is equivalent to the linearized one. Hence, there exist initial conditions 
 $q, \overline{q} \in S\setminus \{p\}$ such that $\displaystyle \lim_{s \to -\infty}\psi(s,q)=p $ and  $ \displaystyle \lim_{s \to +\infty}\psi(s,\overline{q})=p $. 
We define the  unstable and stable sets as  

\begin{equation}\label{w}
\displaystyle W^{u}(p)=\{q \in S: \lim_{s \to -\infty}\psi(s,q)=p\}\,\,\;\; \text{and }\,\,\;\;
\displaystyle W^{s}(p)=\{q\in S: \lim_{s \to +\infty}\psi(s,q)=p\}.
\end{equation}
From  Lemma \ref{l1} we know that, if the function $\tau$ is a non zero constant,i.e.,  $\tau(s)=b$, $b\in \R\setminus\{0\}$ for all $s \in I$, then $b^{2}=a=1$. Our next result provides  two non trivial solutions of the system \eqref{s1}, $a=b^{2}=1$, defined on $\R$, with initial conditions on the set $S$. They are particular cases of the solutions
obtained in  Lemma \ref{l1} with the constant of integration  being zero. Moreover, we also obtain the soliton solutions corresponding to these soltions.

\begin{proposition}\label{pex}
	Let $\psi(s)=(\alpha(s),\tau(s),\eta(s))$ be a solution of \eqref{s1} defined on the maximal interval $I$, $a=1$ and $\psi(0)\in S$, where $S$ is given in \eqref{c}. Then, $\psi(s)=(-s,-1,-s)$ (resp. $\overline{\psi}(s)=(s,1,-s)$), $s \in I=\R$ satisfy \eqref{s1} with initial condition $(0,-1,0)$ (resp. $(0,1,0)$). Moreover,
	\begin{itemize}
		\item[i)] The curve
		\begin{equation}\label{x}
		X(s)=\left(\frac{1+l^{2}+s^2}{2l}, \frac{l^2-1-s^2}{2l},-s\right),	
		\end{equation}
		where $l>0$ is the soliton solution to the CSF in $\mathbb{H}^{2}$ which corresponds to the solution \mbox{$\psi(s)=(-s,-1,-s)$} of \eqref{s1};
		\item[ii)] The curve $$\overline{X}(s)=\left(\frac{1+l^{2}+s^2}{2l}, \frac{l^2-1-s^2}{2l},s\right),$$ where $l>0$ is the soliton solution to the CSF in $\mathbb{H}^{2}$ which corresponds to the solution \mbox{$\psi(s)=(s,1,-s)$} of \eqref{s1}
	\end{itemize}
\end{proposition}

\begin{proof}
	Straightforward computations show that $\psi(s)$ and $\overline{\psi}(s)$ satisfy \eqref{s1} with initial conditions $(0,-1,0)$ and $(0,1,0)$ respectively. 
	
	{\it i)}  Note that, if a curve $X(s)$ is given by \eqref{x}, then
	\begin{eqnarray*}
		T(s)&=&\left(\frac{s}{l},-\frac{s}{l},1\right),\\
		N(s)&=&X(s)\times T(s)=\left(\frac{-1+l^{2}+s^2}{2l}, \frac{l^2+1-s^2}{2l},-s\right),
	\end{eqnarray*}
	$\alpha(s)=\spn{X(s),(0,0,1)}=s$, $\tau(s)=\spn{T(s),(0,0,1)}=1$ and $\eta(s)=\spn{N(s),(0,0,1)}=-s$. 
	
	{\it ii)} The  proof is similar to the case {\it i)}.  			
\end{proof}

 In  Figure \ref{figs}, we illustrate the  curve $X(s)$ in  $\mathbb{H}^{2}$ 
given by  \eqref{x} with $l=1$.

\begin{figure}[h!]
	\centering
	\includegraphics[height=1.3in]{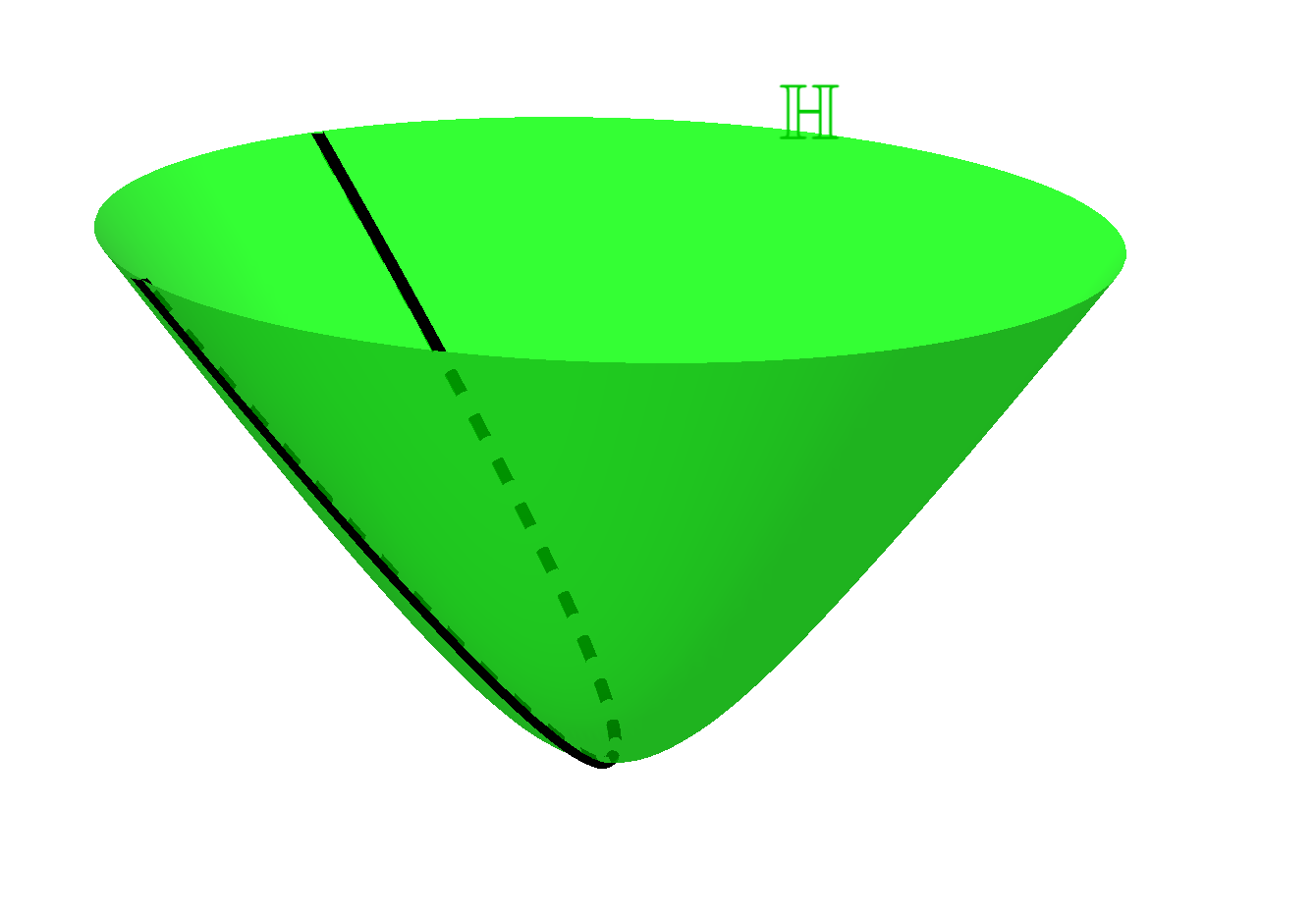}			
	\caption{Soliton solution to the CSF with $a=1$, fixed vector $v=(0,0,1)$ and constant curvature ($k(s)=1$).}
	\label{figs}
\end{figure}

In the following lemmas, we study the behavior of the functions  $\alpha(s)$, $\tau(s)$ and  $\eta(s)$ when  $\psi(s)=(\alpha(s),\tau(s),\eta(s))$ is a non trivial solution of the system  
\eqref{s1} and initial condition $\psi(0)\in S$.

\begin{lemma}\label{ld2}
	Let $\psi(s)=(\alpha(s),\tau(s),\eta(s))$ be a non trivial solution of \eqref{s1} defined on the maximal interval $I=(\omega_{-},\omega_{+})$, $a>0$ and initial condition $\psi(0) \in S$, where $S$ is given in \eqref{c}. If $s_0 \in I$ is a critical point of $\alpha(s)$, then $s_{0}$ is the global minimum (resp. maximum) of $\alpha(s)$ if, and only if, $\alpha(s_{0})>0$ (resp. $\alpha(s_{0})<0$). Moreover, there exists always  $\overline{s}$ such that the function $\displaystyle \alpha(s)$ is monotone on the intervals $(\omega_{-},\overline{s})$ and $(\overline{s}, \omega_{+}).$
\end{lemma}

\begin{proof} Let $s_{0}$ be a critical point of $\alpha(s)$, then $\tau(s_0)=0$. It follows from \eqref{s1} that
	\begin{equation}\label{s_0}
	\alpha^{\prime \prime}(s_0)=\alpha(s_0).
	\end{equation}
	Note that, if $\alpha(s_0)=\tau(s_0)=0$ and $\eta(s_0)=1$, then $\alpha(s)$ is constant, because $(0,0,1)$ is a singular (trivial) solution of \eqref{s1}. 
	
	Therefore, it follows from \eqref{s_0} that $s_{0}$ is a local minimum point of $\alpha(s)$ if $\alpha(s_0)>0$. If there is another critical point $s_1$ of $\alpha(s)$ such that $s_0$ and $s_1$ are consecutive, then $\alpha(s_1)>\alpha(s_0)>0$, because $s_0$ is  a local minimum point. Thus, $\alpha^{\prime \prime}(s_1)=\alpha(s_1)>0$ and $s_1$ is a local minimum point $\alpha(s)$, this is a contradiction. Therefore, if $\alpha(s_0)>0$, then $s_0$  is a  global minimum of $\alpha(s)$.
	
	  If $\alpha(s_0)<0$, it follows from \eqref{s_0} that $s_{0}$ a local maximum point of $\alpha(s)$. The proof that $s_0$ is a global maximum point of $\alpha(s)$ is analogue to the previous case. 
	
	Since  the function $\alpha(s)$ has at  most one critical point, there  exists always   $\overline{s}$ such that the function $\displaystyle \alpha(s)$ is monotone on the intervals $(\omega_{-},\overline{s})$ and $(\overline{s}, \omega_{+}).$
\end{proof}	

We observe that if for a solution 
$\psi(s)=(\alpha(s),\tau(s),\eta(s))$ of \eqref{s1} with  $\psi(0)\in S$, the function $\alpha(s)$ does not have a crtitical point, then  $\alpha(s)$ is monotone on $I$. 
Moreover, it follows from  Lemma \ref{ld2}, that for any solution of  \eqref{s1} with $\psi(0) \in W^{u}\cup W^{s}\setminus \{(0,0,1)\}$, the functions $\alpha(s)$ and $\eta(s)$
do not have critical points. 
 
\begin{lemma}\label{ld4}
	Let $\psi(s)=(\alpha(s),\tau(s),\eta(s))$ be a non trivial solution of \eqref{s1} defined on the maximal interval $I=(\omega_{-},\omega_{+})$, $a>0$ and initial condition $\psi(0) \in S$, where $S$ is given by \eqref{c}. Consider $W^{u}(p)$ and $W^{s}(p)$ given by \eqref{w}. If $\psi(0) \in S\setminus W^{u}(p)$ (resp. $ S\setminus W^{s}(p)$), then $\displaystyle \lim_{s \to \omega_{-}}|\alpha(s)|=+\infty$ (resp. $\displaystyle \lim_{s \to \omega_{+}}|\alpha(s)|=+\infty$).   
\end{lemma}

\begin{proof}
	It follows from Lemma \ref{ld2} that there exits $\overline{s} \in I$ such that $\alpha(s)$ is monotone on the intervals $(\omega_{-},\overline{s})$ and $(\overline{s}, \omega_{+})$. If $\psi(0) \in S\setminus W^{u}(p)$ assume by contradiction that $\alpha(s)$ is bounded  on $(\omega_{-},\overline{s})$. Since $\tau^{2}(s)+\eta^{2}(s)=\alpha^{2}(s)+1$, it follows that the functions $\alpha(s)$ and $\eta(s)$ are bounded and monotone on  $(\omega_{-},\overline{s})$ and the limit $\displaystyle \lim_{s \to \omega_{-}} \tau(s)$ exists.  
	Hence, there exists $q \in \R_{1}^{3}$ such that $\displaystyle \lim_{s \to \omega_{-}}(\alpha(s),\tau(s),\alpha(s))=q$, $\omega_{-}=-\infty$ and $q$ is a singular solution of \eqref{s1}. But the system \eqref{s1} does not have any singular solution on the set $S\setminus W^{u}(p)$. Therefore, $\displaystyle \lim_{s \to \omega_{-}}|\alpha(s)|=+\infty$.
	Similarly, when $\psi(0) \in S\setminus W^{s}(p)$	 one proves that $\displaystyle \lim_{s \to \omega_{+}}|\alpha(s)|=+\infty$.  
\end{proof}

\begin{lemma}\label{ld2a}
	Let $\psi(s)=(\alpha(s),\tau(s),\eta(s))$ be a non trivial solution of \eqref{s1} defined on the maximal interval $I=(\omega_{-},\omega_{+})$, $a>0$ and initial condition $\psi(0) \in S$, where $S$ is given in \eqref{c}. 
	\begin{itemize}
		\item[i)] If $s_0$ is a critical point of $\tau(s)$. Then $\tau^{2}(s_0)\neq 1$ and $a^2\tau^{2}(s_0)\neq 1$. If $\tau(s_0)>0$, then $s_0$ is a local minimum (resp. maximum) point of $\tau(s)$ if, and only if, $\tau(s_0)>1$ (resp. $0<\tau(s_0)<1$). If $\tau(s_0)<0$, then $s_0$ is a local minimum (resp. maximum) point of $\tau(s)$ if, and only if, $\tau(s_0)<-1$ (resp. $-1<\tau(s_0)<0$).
		\item[ii)] The function $\tau(s)$ has at the most a finite number of critical points.		\end{itemize}
\end{lemma}

\begin{proof} {\it i)} 
	Let $s_0$ be a critical point of $\tau(s)$. If $\tau^{2}(s_0)=1$, it follows from $-\alpha^{2}(s)+\tau^{2}(s)+\eta^{2}(s)=1$ that $\alpha^{2}(s_0)=\eta^{2}(s_0)$. Thus, $0=\tau^{\prime}(s_0)=\pm a\eta(s_0)+\alpha(s_0)$ i.e. $a=1$. Hence, it follows from Lemma \ref{l1} that the solution $\psi(s)$ of \eqref{s1} with initial condition $\psi(s_0)$ is a trivial solution, which contradicts the hypothesis. Therefore, $\tau(s_0)\neq 1$. If $a^2\tau^{2}(s_0)= 1$, then $0=\tau^{\prime}(s_0)=\pm\eta(s_0)+\alpha(s_0)$ and from  $-\alpha^{2}(s)+\tau^{2}(s)+\eta^{2}(s)=1$ we have that $\tau(s_0)=1$, which also contradicts the hypothesis. Hence, $\tau^{2}(s_0)\neq 1$ and $a^2\tau^{2}(s_0)\neq 1$. Moreover, taking the second derivative of $\tau(s)$ at $s=s_0$, we obtain \eqref{t}.  This concludes the proof of the item {\it i)}.
	
	{\it ii)}  Note that, it follows from Lemma \ref{ld2} that there exists always  $\overline{s}$ such that $\tau(s)$ does not change sign on each interval $(\omega_{-},\overline{s})$ and $(\overline{s},\omega_{+})$. We will prove for the interval $(\overline{s},\omega_{+})$, since similar arguments can be used for the  interval $(\omega_{-},\overline{s})$. If $\psi(0) \in W^{s}(p)$ i.e. $\displaystyle \lim_{s \to +\infty}\tau(s)=0$, then it follows from item {\it i)} that $\tau(s)$ has at  most a finite number of critical points on $(\overline{s},\omega_{+})$. If $\psi(0) \in S\setminus W^{s}(p)$, then it follows from Lemma \ref{ld4} that $\displaystyle \lim_{s \to \omega_{+}}|\alpha(s)|=+\infty$. 
	
	Assume that $\tau(s)>0$ for all $s \in (\overline{s},\omega_{+})$, then there exists $s_1$ such that $\alpha(s)>0$ for all $s>s_1$. If the function $\eta(s)$, which  is monotone, is always positive, then $\tau^{\prime}(s)>0$  for all $s>s_1$, i.e. $\tau$ has no critical 
on $	(s_1,\omega_+)$.  Now, consider $s_1$ such that  $\eta(s)<0$ for all $s>s_1$.  Assume by contradiction that there are $s_2,s_3 \in (s_1,\omega_{+})$, two local maximum points of $\tau(s)$. From item {\it i)},  we obtain that there are $b,d \in (s_2,s_3)$ such that $b<d$, $\tau(b)=\tau(d)=1$, $\tau^{\prime}(b)<0$, $\tau^{\prime}(d)>0$. It follows from 
	$-\alpha^{2}(s)+\tau^{2}(s)+\eta^{2}(s)=1$ that $\eta^{2}(b)=\alpha^{2}(b)$ and $\eta^{2}(d)=\alpha^{2}(d)$. Thus, from $\tau^{\prime}(b)<0$ we obtain $a\eta(b)<-\alpha(b)<0$ i.e. $a>1$ and from $\tau^{\prime}(d)>0$ we conclude $\alpha(d)>-a\eta(d)>0$ i.e. $a<1$, this is a contradiction.  Therefore, $\tau$ has at most one local maximum point on the  interval $(s_1,\omega_{+})$. 
	
	Analogously, assume that $\tau(s)<0$ for all $s \in (\overline{s},\omega_{+})$, then there exists $s_1$ such that $\alpha(s)<0$ for all $s>s_1$. If the monotone function  $\eta(s)$ is always positive, then $\tau^{\prime}(s)<0$ for all $s>s_1$, and hence $\tau$ has no critical point on $(s_1,\omega_+)$. Now, we consider  $s_1$ such that  $\eta(s)<0$ for all $s>s_1$.  Assume by contradiction that there are $s_2,s_3 \in (s_1,\omega_{+})$, two local maximum points of $\tau(s)$. From item {\it i)}  we obtain that there are $b,d \in (s_2,s_3)$ such that $b<d$, $\tau(b)=\tau(d)=-1$, $\tau^{\prime}(b)<0$, $\tau^{\prime}(d)>0$. It follows from $-\alpha^{2}(s)+\tau^{2}(s)+\eta^{2}(s)=1$ that $\eta^{2}(b)=\alpha^{2}(b)$ and $\eta^{2}(d)=\alpha^{2}(d)$. Thus, from $\tau^{\prime}(b)<0$ we obtain that $\alpha(b)<a\eta(b)<0$ i.e $a<1$ and from $\tau^{\prime}(d)>0$ we conclude that $0>\alpha(d)>a\eta(d)$ i.e. $a>1$, this is a contradiction. Therefore, $\tau$ has at most one local maximum point on the  interval $(s_1,\omega_{+})$.
	
Similar  arguments for the interval $(\omega_{-},\overline{s})$ imply that 	
	 $\tau(s)$ has at  most a finite number of critical points on $I$.
\end{proof}

The following lemma shows that  the function $\tau(s)$ is bounded and hence the curvature of the soliton $X(s)$ on $\mathbb{H}^{2}$ is bounded.

\begin{lemma}\label{ld6}
	Let $\psi(s)=(\alpha(s),\tau(s),\eta(s))$ be a non trivial solution of \eqref{s1} defined on the maximal interval $I=(\omega_{-},\omega_{+})$, $a>0$ with initial condition $\psi(0) \in S$, where $S$ is given by \eqref{c}. Then the function $\tau(s)$ is bounded on $I$.
\end{lemma}

\begin{proof}
	It follows from Lemma \ref{ld2} that there exists $\overline{s} \in I$ such that $\alpha(s)$ is monotone on the intervals $(\omega_{-},\overline{s})$ and $(\overline{s}, \omega_{+})$. Moreover, from Proposition \ref{p2} we have that $\eta(s)$ is monotone.´
		
	If $\psi(0) \in  W^{s}(p)$, where $p$ is a singular point, then $\displaystyle \lim_{s \to +\infty}\tau(s)=0$ and $\tau(s)$ is bounded on $(\overline{s},+\infty)$ for any $\overline{s} \in I$ fixed. Similarly, if $\psi(0) \in W^{u}(p)$, then $\displaystyle \lim_{s \to -\infty}\tau(s)=0$ and $\tau(s)$ is bounded on  $(-\infty,\overline{s})$, $\overline{s} \in I$ fixed. 
	
We will now consider the cases when  the initial condition $\psi(0)$ belongs to $S\setminus W^{u}(p)$ or $S\setminus W^{s}(p)$.
	If $\psi(0) \in S\setminus W^{u}(p)$  assume by contradiction that $\tau(s)$ is unbounded on $(\omega_{-},\overline{s})$, then it follows from Lemma \ref{ld2a} that there exists $s_1 \in (\omega_{-},\overline{s})$ such that $|\tau(s)|>1$ and $a|\tau(s)|>2$ for all $s \in (\omega_{-},s_1)$. Thus, $|\alpha(s)|>|\eta(s)|$, because $\alpha^{2}(s)-\eta^{2}(s)=\tau^{2}(s)-1>0$ and $a\tau^{2}(s)>2|\tau(s)|$ for all $s \in (\omega_{-},s_1)$. From Lemma \ref{ld2} we have that $\tau(s)$ does not sign on $(\omega_{-},s_1)$.
	
	If $\alpha^{\prime}(s)=\tau(s)<0$ on $ (\omega_{-},s_1)$, then $\alpha(s)$ is strictly decreasing on this interval. Thus, it follows from Lemma \ref{ld4}, that $\displaystyle \lim_{s \to \omega_{-}}\alpha(s)=+\infty$. Hence, $s_1$ can be chosen so that $\alpha(s)$ is decreasing and positive for all $s<s_1$. Therefore, using \eqref{s1} and the fact that $a\tau^{2}(s)>2|\tau(s)|$, we obtain
	\begin{equation*}
	2\alpha(s)-2\alpha(s_1)=-2\int^{s_1}_{s}\tau(s)ds<\int^{s_1}_{s}a\tau^{2}(s)ds=\eta(s)-\eta(s_1),
	\end{equation*}
	i.e.,
	\begin{equation*} 
	0<\alpha(s)-\eta(s)< 2\alpha(s_1)-\eta(s_1)-\alpha(s),
	\end{equation*}
	which contradicts  Lemma \ref{ld4}, because $\displaystyle \lim_{s \to \omega_{-}}\alpha(s)=+\infty$. Hence, $\tau(s)$ is bounded  on $(\omega_{-},s_1)$. 	
	
	If $\alpha^{\prime}(s)=\tau(s)>0$ on $ (\omega_{-},s_1)$, then $\alpha(s)$ is strictly increasing on this interval. Thus, it follows from Lemma \ref{ld4} that $\displaystyle \lim_{s \to \omega_{-}}\alpha(s)=-\infty$. Hence, $s_1$ can be chosen so that $\alpha(s)$ is increasing and negative for all $s<s_1$. Therefore, using \eqref{s1} and the fact that $a\tau^{2}(s)>2|\tau(s)|$, we obtain	
	\begin{equation*}
	2\alpha(s_1)-2\alpha(s)=2\int^{s_{1}}_{s}\tau(s)ds<\int^{s_{1}}_{s}a\tau^{2}(s)ds=-\eta(s_{1})+\eta(s),
	\end{equation*}
	i.e. ,
	\begin{equation*}
	0<-\alpha(s)-\eta(s)< -2\alpha(s_1)-\eta(s_{1})+\alpha(s),
	\end{equation*}
	which contradicts Lemma \ref{ld4}, because $ \displaystyle \lim_{s \to \omega_{-}}\alpha(s)=-\infty$. Hence, $\tau(s)$ is bounded on $(\omega_{-},\overline{s})$. 
	
	When $\psi(0) \in S\setminus W^{s}(p)$, the similar arguments show  that $\tau(s)$ is bounded on $(\overline{s},\omega_{+})$.  
	
	Therefore, the function $\tau(s)$ is bounded on $I$.  
\end{proof}

Our next lemma  provides the behavior of the function $\eta(s)$.

\begin{lemma}\label{et} 
	Let $\psi(s)=(\alpha(s),\tau(s),\eta(s))$ be a non trivial solution of \eqref{s1} defined on the maximal interval $I=(\omega_{-},\omega_{+})$, $a>0$ and initial condition $\psi(0) \in S$, where $S$ is given by \eqref{c}. Consider $W^{u}(p)$ and $W^{s}(p)$ given by \eqref{w}. If $\psi(0) \in S\setminus W^{u}(p)$ (resp. $S\setminus W^{s}(p)$), then $\displaystyle \lim_{s \to \omega_{-}}\eta(s)=+\infty$ (resp. $\displaystyle \lim_{s \to \omega_{+}}\eta(s)=-\infty$). 
\end{lemma}
\begin{proof}
	The proof follows from  Lemmas \ref{ld4} and \ref{ld6} and the fact that $\tau^2(s)+\eta^2(s)=\alpha^2(s)+1$.    
\end{proof}

\begin{lemma}\label{ld7}
	Let $\psi(s)=(\alpha(s),\tau(s),\eta(s))$ be a non trivial solution of \eqref{s1} defined on the maximal interval $I=(\omega_{-},\omega_{+})$, $a>0$ and initial condition $\psi(0) \in S$, where $S$ is given by \eqref{c}. Then $I=\R$. 
\end{lemma}

\begin{proof}
	It follows from Lemma \ref{ld6} that $\tau(s)$ is bounded on the interval $I$. Let $M>0$ be such that $|\tau(s)|\leq M$ for all $s \in I$. Using \eqref{s1}, we obtain that
	\begin{equation}\label{e24}
	|\alpha(s)-\alpha(s_{0})|=\left|\int_{s_{0}}^{s}\tau(s)ds\right|\leq M|s-s_{0}|
	\end{equation}
	for all $s \in I$. We will now show that $\omega_{-}=-\infty$ and $\omega_{+}=+\infty$.
	
	If $\psi(0) \in S\setminus W^{u}(p)$, then from Lemma \ref{ld4} we have that $\alpha(s)$ is unbounded on $(\omega_{-},\overline{s})$ for any $\overline{s} \in I$ fixed. Hence, it follows from \eqref{e24} that $\omega_{-}=-\infty$. It follows from the definition of $W^{u}(p)$ that $\omega_{-}=-\infty$ when $\psi(0) \in W^{u}(p)$. Since $\displaystyle S= W^{u}(p)\cup \left[S\setminus W^{u}(p)\right]$, we conclude that $\omega_{-}=-\infty$.
	
	If $\psi(0) \in S\setminus W^{s}(p)$, then from Lemma \ref{ld4} we have that $\alpha(s)$ is unbounded on $(\overline{s},\omega_{+})$ for any $\overline{s} \in I$ fixed. Hence, it follows from \eqref{e24} that $\omega_{+}=+\infty$. It follows from the definition of $W^{s}(p)$ that $\omega_{+}=+\infty$ when $\psi(0) \in W^{s}(p)$. Since $\displaystyle S= W^{s}(p)\cup \left[S\setminus W^{s}(p)\right]$, then $\omega_{+}=+\infty$.
		
	Therefore, $I=\R$.
\end{proof}

\begin{lemma} \label{coro1}
	Let $\psi(s)=(\alpha(s),\tau(s),\eta(s))$ be a non trivial solution of \eqref{s1}, with $a>0$ and initial condition $\psi(0) \in H\cup C \cup S$, where $H$, $C$ and $S$ are given by \eqref{c}. Then $\psi(s)$ and the corresponding soliton solution $X(s)$ to the CSF on  $\mathbb{H}^{2}$ are defined for all  $s\in \R$. Moreover, at each end the curvature $k(s)$ of $X$ converges to one of the following constants $\{-1,0,1\}$.  
\end{lemma}

\begin{proof}
	Since $X(s)$ is a soliton solution to the CSF corresponding to $\psi(s)$,  then $k(s)=a\tau(s)$. Thus,  Lemmas \ref{l3}, \ref{l7}, \ref{ld2} and \ref{ld6} imply  that $k(s)$ is bounded on $\R$ and it has at  most a finite number of critical points. Thus, the limits  $\displaystyle \lim_{s \to \pm\infty}k(s)=\displaystyle \lim_{s \to \pm\infty}a\tau(s)$ exist. In particular, when $\psi(0)\in W^u(p)$ then $\displaystyle \lim_{s \to -\infty} \tau(s)=0$. Similarly, when $\psi(0)\in W^s(p)$ then $\displaystyle \lim_{s \to +\infty} \tau(s)=0$. In these cases, the curvature function converges to zero at $-\infty$ and 
	$+\infty$, respectively.
	 
 If $\displaystyle \lim_{s \to \pm\infty} \tau(s)\neq 0$, then $\displaystyle \lim_{s \to \pm\infty} |\alpha(s)|=+\infty$ and it follows from $-\alpha^{2}(s)+\tau^{2}(s)+\eta^{2}(s)=\delta$, where $\delta \in \{-1,0,1\}$ that $$\displaystyle \lim_{s \to \pm\infty} \frac{\eta^{2}(s)}{\alpha^{2}(s)}=\displaystyle \lim_{s \to \pm\infty} \left(\frac{-\tau^{2}(s)+\delta}{\alpha^{2}(s)}+1\right)=1.$$
	Using  \eqref{s1},  Lemmas \ref{l5}, \ref{l6}, \ref{ld4}, \ref{et} and  L'Hospital rule, we obtain $$\displaystyle \lim_{s \to \pm\infty} -\frac{\eta(s)}{\alpha(s)}=\lim_{s \to \pm\infty}a\tau(s)=\lim_{s \to \pm\infty}k(s).$$
	Therefore, $\displaystyle \lim_{s \to -\infty}k(s)=\pm 1$ and $\displaystyle \lim_{s \to +\infty}k(s)=\pm1$.
\end{proof}

\vspace {.2in}

Finally, we will prove our main theorem.

\vspace{.2in}

\noindent {\it Proof of Theorem \ref{c2t2}}. For any vector $v\in \R^3_1\setminus \{0\}$,   without loss generality we can consider $v=ae$, where $a>0$ and 
		\begin{eqnarray*}
		e=	\left\{\begin{array}{clll}
			(-1,0,0) & \text{if} & v & \text{is a timelike vector},\\
			(-1,1,0) & \text{if} & v &  \text{is a lightlike vector },\\
			(0,0,1) & \text{if} & v &  \text{is a spacelike vector.}
		\end{array} \right.
	\end{eqnarray*}	 
	Let $\psi(s)=(\alpha(s),\tau(s),\eta(s))$ be a solution of \eqref{s1} defined on the maximal interval $I=(\omega_{-},\omega_{+})$, $a>0$ and initial condition $\psi(0) \in \R^{3}$ satisfying  		 
	\begin{eqnarray*}
		-\alpha^{2}(0)+\tau^{2}(0)+\eta^{2}(0)=	\left\{
		\begin{array}{clll}
			-1 & \text{if} & v & \text{is a timelike vector},\\
			0 & \text{if} & v &   \text{is a lightlike vector},\\
			1 & \text{if} & v &  \text{is a spacelike vector,}
		\end{array} \right.
	\end{eqnarray*}
i.e., $\psi(0) \in H\cup C \cup S$, where $H$, $C$ and $S$ are the disjoint sets given by \eqref{c}.  
Moreover,  it follows from Proposition \ref{p2} that there is a soliton solution $X(s)$ to the CSF, with curvature $k(s)=a\tau(s)$, such that the relations 
 $$\alpha(s)=\spn{X(s),e},\hspace{0.3 cm} \tau(s)=\spn{T(s),e}\hspace{0.3 cm}\text{and} \hspace{0.3 cm} \eta(s)=\spn{N(s),e}, $$
 are satisfied, where $T$ and $N$ are the unit vector fields tangent and normal to $X$.

	Thus, the initial conditions of \eqref{s1},  which are given by  two constants,  determine the soliton solution in each case. Therefore, for each fixed vector $v \in \R_{1}^{3}\setminus \{0\}$ there is a 2-parameter family of non trivial soliton solutions to the CSF in $\mathbb{H}^{2}$. 
	Moreover, 	it follows from Lemmas \ref{l1}, \ref{lb8} and \ref{ld7} that each  soliton solution is defined for all $s \in \R$, i.e. $I=\R$ and Lemma \ref{coro1} shows that  the curvature at each end converges to one of the following constants $\{-1, 0, 1\}$.  
	
	Note that, from Lemmas \ref{l1}, \ref{l2} and \ref{ld2} we know  that there exists $\overline{s} \in \R$ such that $\alpha(s)$ is strictly monotone on the  intervals $(-\infty,\overline{s})$ and $(\overline{s}, +\infty)$. Since $\alpha(s)$ describes the Euclidean height of $X(s)$ with respect to a fixed plane, then $X(s)$ does not have self-intersections  in each one of the intervals$(-\infty,\overline{s})$ and $(\overline{s}, +\infty)$. Therefore,  $X(s)$ is embedded if $\alpha(s)$ is monotone in $\R$.
	
	If $\alpha(s)$ is not monotone in $\R$ then $\alpha(s)$ has only one critical point.  Suppose that $X(s)$ has some self-intersection and consider $\Sigma$ the simple region bounded by $X([s_1,s_2])$ with $X(s_1)=X(s_2)$, $s_1<\overline{s}<s_2$ and $\theta$ the external angle between the tangent vectors $T(s_1)$ and $T(s_2)$, which is at the most $\pi$. By Gauss-Bonnet's theorem, we obtain
	
	\begin{eqnarray*}
		0<2\pi\chi(\Sigma)-\theta&=&\int_{\Sigma}\kappa d\sigma+\int_{X([s_1,s_2])}k(s)ds\\
		&=&-\int_{\Sigma} d\sigma +\int_{X([s_1,s_2])}a\tau(s)ds\\
		&=&-\int_{\Sigma} d\sigma +a[\alpha(s_2)-\alpha(s_1)]\\
		&=&-\int_{\Sigma} d\sigma<0.
	\end{eqnarray*} 
	This is a contradiction. Hence, the soliton solution $X(s)$ to the CSF in $\mathbb{H}^{2}$ does not admit self-intersections. Note that, $X(s)$ is already  embedded on the intervals $(-\infty,\overline{s})$ and $(\overline{s}, +\infty)$ and from Lemmas \ref{l5} and \ref{ld4} we have that the two ends of the curve are unbounded. Therefore, $X(s)$ is an embedded curve. \hfill $\square$

\section{Visualizing some Soliton Solutions to the CSF on $\mathbb{H}^{2}$}	
In this setion, we visualize some examples of soliton solutions to the CSF on the hyperbolic space. In order to do so, we use the following parametrization.
for $\mathbb{H}^{2}$ 
\begin{equation*}
\chi(u,w)=(\sqrt{1+u^2+w^2},u,w)
\end{equation*}
If a curve of $\mathbb{H}^{2}$ $X(s)=\chi(u(s),w(s))$ is parametrized by arc lentgh, then  
\begin{equation*}
T(s)=\left(\frac{u(s)u^{\prime}(s)+w(s)w^{\prime}(s)}{\sqrt{1+u^{2}(s)+w^{2}(s)}},u^{\prime}(s),w^{\prime}(s)\right)
\end{equation*}
the functions 
$u(s)$ and  $w(s)$ satisfy the following system of ODEs 
\begin{equation}\label{ex1}
\left\{\begin{array}{ll}
(u^{\prime})^{2}+(w^{\prime})^{2}+\left(u^{\prime}w-uw^{\prime}\right)^{2}=1+u^2+w^2,\\		   
w^{\prime \prime}u^{\prime}-u^{\prime \prime}w^{\prime} +uw^{\prime}-u^{\prime}w=k(s)\sqrt{1+u^{2}+w^{2}},
\end{array}\right.		
\end{equation}
where $k(s)$ is the curvature  of $X(s)$. The first equation follows from the fact that 
the curve is parametrized by arc lentgh and the second one from the expression of the curvature of $X$.

In Theorem  \ref{c2t1}, we saw that the curvature of a soliton solution to the CSF on  $\mathbb{H}^{2}$ is determined by its tangent vector field and a non zero fixed vector $v$.
We use  (\ref{ex1}) and  the {\it software  Maple} to plot examples of such solitons. 
In each example, we visualize the curve on the three models of the 2-dimensional hyperbolic space, namely the hyperboloid, the Poincar\'e disk and  the upper half space.   

In Figure \ref{fig5} a), the blue curve on the hyperboloid provides the visualization of a soliton  solution $X(s)$ to the CSF on $\mathbb{H}^{2}$ whose curvature is given by  $\displaystyle k(s)=\spn{T(s),(-1,0,0)}$ and $a=1$. The red curve  is the Euclidean orthogonal projection of $X(s)$ on the plane that contains the origin and it is orthogonal to the vector $(-1,0,0)$. In Figures \ref{fig5} b) and  c) we visualize the same soliton on the Poincar\'e disk and on the half space model respectively. 

\begin{figure}[!htb]
	\centering
	\subfloat[]{
		\includegraphics[height=3 cm]{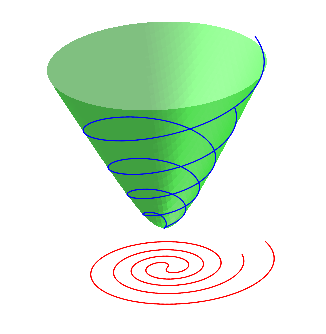}
	}
	\quad 
	\subfloat[]{
		\includegraphics[height=3 cm]{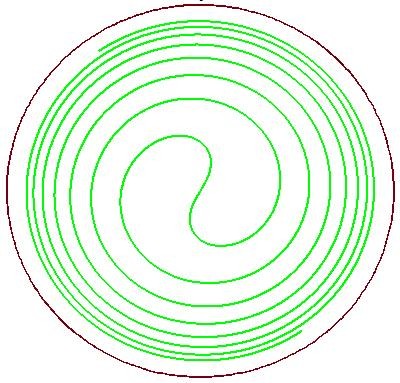}
	}
	\quad 
	\subfloat[]{
		\includegraphics[height= 3 cm]{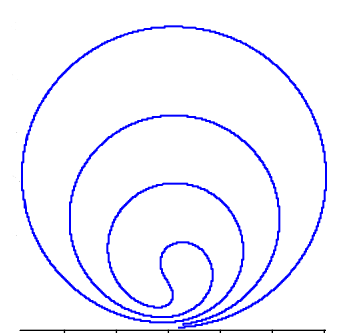}
	}
	\caption{Soliton solution to the CSF on  $\mathbb{H}^{2}$ with fixed vector $v=(-1,0,0)$ and  $a=1$.}
	\label{fig5}
\end{figure}
In Figure \ref{fig6} a), the blue curve provides the visualization  
of a soliton  solution $X(s)$ to the CSF on $\mathbb{H}^{2}$ whose curvature is given by  $\displaystyle k(s)=\spn{T(s),(-1,1,0)}$ and $a=1$.  In Figures \ref{fig6} b) and  c) we visualize the same soliton on the Poincar\'e disk and on the half space model respectively. 

\begin{figure}[!htb]
	\centering
	\subfloat[]{
		\includegraphics[height=3 cm]{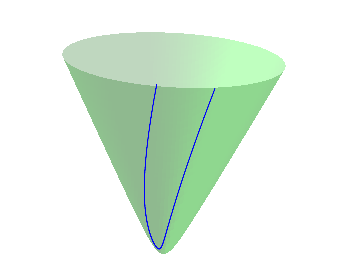}
	}
	\quad 
	\subfloat[]{
		\includegraphics[height=3 cm]{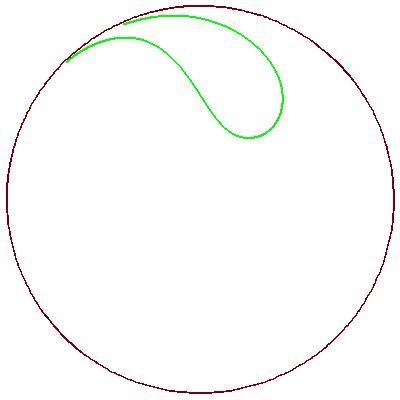}
	}
	\quad 
	\subfloat[]{
		\includegraphics[height= 3 cm]{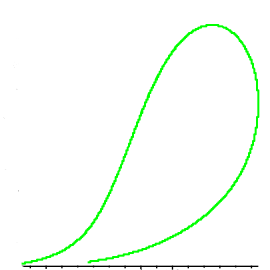}
	}
	\caption{Sóliton do fluxo FRC em $\mathbb{H}^{2}$ com vetor fixado $(-1,1,0)$ e $a=1$.}
	\label{fig6}
\end{figure}


Finally, in Figure  \ref{fig7} a), the blue curve provides the visualization  
of a soliton  solution $X(s)$ to the CSF on $\mathbb{H}^{2}$ whose curvature is given by  $\displaystyle k(s)=\spn{T(s),(0,0,1)}$ and $a=1$.  In Figures \ref{fig6} b) and  c) we visualize the same soliton on the Poincar\'e disk and on the half space model respectively. 
We point out that this soliton has non constant curvature and hence it is different from 
the one given in Proposition \ref{pex}. In fact, in order to obtain 
Figure \ref{fig7}, we used initial condition 
 $u(0)=w(0)=0$ and $\displaystyle u^{\prime}(0)=-w^{\prime}(0)=-\,\frac{1}{\sqrt{2}}$ for the system  \eqref{ex1}, i.e., $\displaystyle T(0,0)=\left(0,-\frac{1}{\sqrt{2}},\frac{1}{\sqrt{2}}\right)$ and  $ \displaystyle \tau(0)=\frac{1}{\sqrt{2}}\neq \pm1$. 
Hence the curvature is not constant. 
In Figures \ref{fig7} b) and  c) we visualize the soliton given in Figure \ref{fig7} a)
on the Poincar\'e disk and on the half space model respectively. 

\begin{figure}[htb!]
	\centering
	\subfloat[]{
		\includegraphics[height=3 cm]{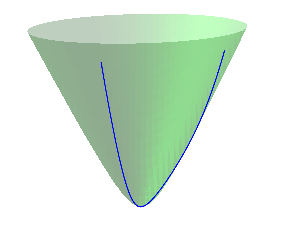}
	}
	\quad 
	\subfloat[]{
		\includegraphics[height=3 cm]{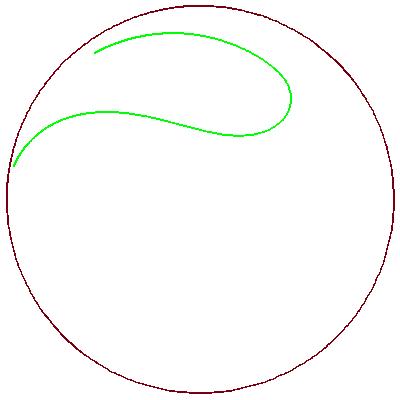}
	}
	\quad 
	\subfloat[]{
		\includegraphics[height= 3 cm]{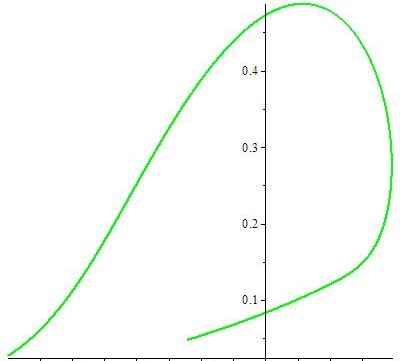}
	}
	\caption{Sóliton do fluxo FRC em $\mathbb{H}^{2}$ com vetor fixado $(0,0,1)$ e $a=1$.}
	\label{fig7}
\end{figure}


\bf Aknowlegment: \rm The first author aknowledges the support given by the Universidade 
Federal do Oeste da Bahia during his graduate program at the Universidade de Bras\'\i lia, when this research was undertaken.

\end{document}